\documentclass[11pt]{article}
\usepackage{latexsym,amsfonts,amssymb,amsmath,amsthm}

\usepackage{color,comment}

\begin{document}

\newcommand{ \bl}{\color{blue}}
\newcommand {\rd}{\color{red}}
\newcommand{ \bk}{\color{black}}
\newcommand{ \gr}{\color{OliveGreen}}
\newcommand{ \mg}{\color{RedViolet}}

\newcommand{\norm}[1]{||#1||}
\newcommand{\normo}[1]{|#1|}
\newcommand{\secn}[1]{\addtocounter{section}{1}\par\medskip\noindent
 {\large \bf \thesection. #1}\par\medskip\setcounter{equation}{0}}

\newcommand{\secs}[1]{\addtocounter {section}{1}\par\medskip\noindent
 {\large \bf  #1}\par\medskip\setcounter{equation}{0}}
\renewcommand{\theequation}{\thesection.\arabic{equation}}

\setlength{\baselineskip}{16pt}

\newtheorem{theorem}{Theorem}[section]
\newtheorem{lemma}{Lemma}[section]
\newtheorem{proposition}{Proposition}[section]
\newtheorem{definition}{Definition}[section]
\newtheorem{example}{Example}[section]
\newtheorem{corollary}{Corollary}[section]

\newtheorem{remark}{Remark}[section]

\numberwithin{equation}{section}

\def\p{\partial}
\def\I{\textit}
\def\R{\mathbb R}
\def\C{\mathbb C}
\def\u{\underline}
\def\l{\lambda}
\def\a{\alpha}
\def\O{\Omega}
\def\e{\epsilon}
\def\ls{\lambda^*}
\def\D{\displaystyle}
\def\wyx{ \frac{w(y,t)}{w(x,t)}}
\def\imp{\Rightarrow}
\def\tE{\tilde E}
\def\tX{\tilde X}
\def\tH{\tilde H}
\def\tu{\tilde u}
\def\d{\mathcal D}
\def\aa{\mathcal A}
\def\DH{\mathcal D(\tH)}
\def\bE{\bar E}
\def\bH{\bar H}
\def\M{\mathcal M}
\renewcommand{\labelenumi}{(\arabic{enumi})}

\def\disp{\displaystyle}
\def\undertex#1{$\underline{\hbox{#1}}$}
\def\card{\mathop{\hbox{card}}}
\def\sgn{\mathop{\hbox{sgn}}}
\def\exp{\mathop{\hbox{exp}}}
\def\OFP{(\Omega,{\cal F},\PP)}
\newcommand\JM{Mierczy\'nski}
\newcommand\RR{\ensuremath{\mathbb{R}}}
\newcommand\CC{\ensuremath{\mathbb{C}}}
\newcommand\QQ{\ensuremath{\mathbb{Q}}}
\newcommand\ZZ{\ensuremath{\mathbb{Z}}}
\newcommand\NN{\ensuremath{\mathbb{N}}}
\newcommand\PP{\ensuremath{\mathbb{P}}}
\newcommand\abs[1]{\ensuremath{\lvert#1\rvert}}

\newcommand\normf[1]{\ensuremath{\lVert#1\rVert_{f}}}
\newcommand\normfRb[1]{\ensuremath{\lVert#1\rVert_{f,R_b}}}
\newcommand\normfRbone[1]{\ensuremath{\lVert#1\rVert_{f, R_{b_1}}}}
\newcommand\normfRbtwo[1]{\ensuremath{\lVert#1\rVert_{f,R_{b_2}}}}
\newcommand\normtwo[1]{\ensuremath{\lVert#1\rVert_{2}}}
\newcommand\norminfty[1]{\ensuremath{\lVert#1\rVert_{\infty}}}
\newcommand{\ds}{\displaystyle}

\title{Persistence, coexistence and extinction in two species chemotaxis models on bounded heterogeneous environments}
\author{
Tahir  Bachar Issa and Wenxian Shen  \\
Department of Mathematics and Statistics\\
Auburn University\\
Auburn University, AL 36849\\
U.S.A. }

\date{}
\maketitle

\noindent {\bf Abstract.} { { In this paper, we} consider  two species chemotaxis systems with Lotka-Volterra type competition terms in heterogeneous media. We first find various conditions on the parameters which guarantee the global existence and boundedness of classical solutions with nonnegative initial functions. Next, we find further conditions on the parameters which establish the persistence of the two species. Then, under the same set of conditions  for the persistence of two species, we  prove the existence of coexistence states.  Finally we prove the extinction phenomena in the sense that one of the species dies out asymptotically and the other reaches its carrying capacity as time goes to infinity.  The persistence in general two species chemotaxis systems is studied for the first time.   Several important techniques are developed to study the persistence and coexistence of  two species chemotaxis systems. Many existing results on the persistence, coexistence, and extinction  on two species competition systems without chemotaxis are recovered.
}
\medskip


\medskip


\section{Introduction and the statements of the main results}
\label{S:intro}
Chemotaxis, the tendency of cells, bacteria, or   organisms to orient the direction of their movements toward the increasing or decreasing concentration of a signaling chemical substance,  has a crucial role in a wide range of biological phenomena such as immune system response, embryo development, tumor growth, etc. (see  \cite{ISM04}).  Recent studies describe also macroscopic process such as population dynamics or gravitational collapse, etc.,  in terms of chemotaxis  (see \cite{DAL1991}).  Because of its  crucial  role in the above mentioned process and others, chemotaxis has attracted great attention in both   biological and mathematical communities.  In 1970's, Keller and Segel  proposed in \cite{KS1970, KS71}  a celebrated mathematical model (K-S model) made up of two parabolic equations to describe chemotaxis. Since their pioneering works, a large amount of work has been devoted to determine whether solutions exists globally or when blow-up occurs in the K-S model and its various forms of variants. For a broad survey on the progress of various chemotaxis models and a rich selection of references, we refer the reader  to the survey papers \cite{NBYTMW05, THKJP09, H03, Win2013}. But many fundamental problems are still not well understood yet.  In particular, there is little study on chemotaxis systems with time and space dependent logistic sources.

 In reality, the environments of many living organisms are spatially and temporally heterogeneous. It is of both biological and mathematical interests to study chemotaxis models with certain time and space dependence. To the best of our knowledge, the first paper on chemotaxis systems with time and space dependent logistic sources is our paper \cite{ITBWS16}, where we considered a one species parabolic-elliptic chemotaxis
 model and  proved, under the well known  assumption of  smallness of chemotaxis effect comparing  to logistic damping effect, the global existence and boundedness of nonnegative classical solutions,
  the existence of  positive entire solutions,  and under some further conditions, the uniqueness and nonlinear stability of entire solutions.

  In the current paper, we consider the following two species parabolic-parabolic-elliptic chemotaxis system with heterogeneous Lotka-Volterra type competition terms,
 \begin{equation}
 \label{u-v-w-eq00}
\begin{cases}
u_t=d_1\Delta u-\chi_1\nabla\cdot (u \nabla w)+u\Big(a_0(t,x)-a_1(t,x)u-a_2(t,x)v\Big),\quad x\in \Omega\cr
v_t=d_2\Delta v-\chi_2\nabla \cdot(v \nabla w)+v\Big(b_0(t,x)-b_1(t,x)u-b_2(t,x)v\Big),\quad x\in \Omega\cr
0=d_3\Delta w+k u+lv-\lambda w,\quad x\in \Omega \cr
\frac{\p u}{\p n}=\frac{\p v}{\p n}=\frac{\p w}{\p n}=0,\quad x\in\p\Omega,
\end{cases}
 \end{equation}
where $\Omega \subset \mathbb{R}^n(n\geq 1)$ is a bounded domain with smooth boundary,  $d_i$ ($i=1,2,3$) are positive constants, $\chi_1,\chi_2,k,l,\lambda$ are nonnegative constants, and $a_i(t,x)$ and $b_i(t,x)$ ($i=0,1,2$) are positive bounded smooth functions.

 Note that, in the absence of chemotaxis, that is,  $\chi_1=\chi_2=0$, the dynamics of \eqref{u-v-w-eq00} is determined by the first two equations, that is, the following two species competition system,
\begin{equation}
\begin{cases}
\label{u-v-eq00}
u_t=d_1\Delta u+u\Big(a_0(t,x)-a_1(t,x)u-a_2(t,x)v\Big),\quad x\in \Omega\cr
v_t=d_2\Delta v+v\Big(b_0(t,x)-b_1(t,x)u-b_2(t,x)v\Big),\quad x\in \Omega\cr
\frac{\p u}{\p n}=\frac{\p v}{\p n}=0,\quad x\in\p\Omega.
\end{cases}
 \end{equation}
Among interesting dynamical issues  in \eqref{u-v-w-eq00} and \eqref{u-v-eq00} are  persistence, coexistence, and extinction.
These dynamical issues for \eqref{u-v-eq00} have been extensively studied (see  \cite{Ahm}, \cite{FuMa97}, \cite{HeSh}, \cite{HeSh02},  etc.).
Several authors have studied these issues for system \eqref{u-v-w-eq00} with constant coefficients \cite{TBJLMM16, ITBRS17,   NT13, STW13, TW12}. For example in \cite{ITBRS17}, the authors considered a more general competitive-cooperative chemotaxis system with nonlocal terms logistic sources and proved both the phenomena of coexistence and of exclusion for parameters in some natural range.
 However, there is little study of these important issues for \eqref{u-v-w-eq00} with time and space dependent coefficients. The objective of this
paper is to investigate the persistence, coexistence, and extinction dynamics of \eqref{u-v-eq00}. In particular, we
 identify the circumstances  under which
 persistence or extinction occurs, and in the case that persistence occurs, we study the existence of coexistence states.

 In order to do so, we  first study the global existence of classical solutions of \eqref{u-v-w-eq00} with any given nonnegative initial functions.  Let
 $$ { C^+(\bar{\Omega})=\left\{u\in C(\bar{\Omega})  \,|\, u \geq 0 \right \}.}
   $$
   Note that for any given $t_0\in\RR$ and  $u_0,v_0 \in C^+(\bar{\Omega})$, $\eqref{u-v-eq00}$  has a unique  bounded  global  classical solution $(u(x,t;t_0,u_0,v_0),v(x,t;t_0,u_0,v_0))$ satisfying that
\begin{equation}
\label{ic}
(u(x,t_0;t_0,u_0,v_0),v(x,t_0;t_0,u_0,v_0))=(u_0(x),v_0(x)).
\end{equation}
 However, it is not known whether for any given $t_0\in\RR$ and  $u_0,v_0 \in { C^+(\bar{\Omega})}$,
\eqref{u-v-w-eq00} has a unique  bounded  global  classical solution $(u(x,t;t_0,u_0,v_0),v(x,t;t_0,u_0,v_0),w(x,t;t_0,u_0,v_0))$
  satisfying  \eqref{ic}. { It can be proved that for given $u_0,v_0 \in C^+(\bar{\Omega})$ and $t_0 \in \RR,$
  \eqref{u-v-w-eq00} has a unique  local  classical solution $(u(x,t;t_0,u_0,v_0),v(x,t;t_0,u_0,v_0),w(x,t;t_0,u_0,v_0))$
  satisfying \eqref{ic} (see Lemma \ref{lm-local-001}). If no confusion occurs,
  we denote  $(u(x,t;t_0,u_0,v_0),v(x,t;t_0,u_0,v_0),w(x,t;t_0,u_0,v_0))$ by $(u(x,t;t_0),v(x,t;t_0),w(x,t;t_0))$.} To formulate our results on global existence of classical solutions
of \eqref{u-v-w-eq00}, we introduce the following notations. For  a given function $f_i(t,x)$ defined on $\mathbb{R} \times \bar \Omega$ we put
 \begin{equation*}
 \label{f-i-sup-inf-eq1}
f_{i,\inf}=\inf _{ t \in\RR,x \in\bar{\Omega}}f_i(t,x),\quad f_{i,\sup}=\sup _{ t \in\RR,x \in\bar{\Omega}}f_i(t,x),
\end{equation*}
 \begin{equation*}
 \label{f-i-sup-inf-eq2}
f_{i,\inf}(t)=\inf _{x \in\bar{\Omega}}f_i(t,x),\quad f_{i,\sup}(t)=\sup _{x \in\bar{\Omega}}f_i(t,x),
 \end{equation*}
 unless specified otherwise.
 We also introduce the following assumptions.

 \smallskip
\noindent {\bf (H1)} {\it $a_i(t,x)$, $b_i(t,x)$, $\chi_i$  and  $d_3$, $k$ and $l$ satisfy
\begin{equation}
\label{global-existence-cond-eq00}
a_{1,\inf}>\frac{ k\chi_1}{d_3},\quad a_{2,\inf}\geq \frac{l \chi_1}{d_3}, \quad b_{1,\inf}\geq \frac{k \chi_2}{d_3},\quad \text{and} \quad b_{2,\inf}>\frac{ l\chi_2}{d_3}.
 \end{equation}
}

\noindent {\bf (H2)} {\it $a_i(t,x)$, $b_i(t,x)$, $\chi_i$  and  $d_3$, $k$ and $l$ satisfy
 \begin{equation}
\label{global-existence-cond-eq01}
a_{1,\inf}>\frac{ k\chi_1}{d_3},\quad  b_{2,\inf}>\frac{l \chi_2}{d_3},\quad {\rm and}\quad
\big(a_{1,\inf}-\frac{ k\chi_1}{d_3}\big)\big( b_{2,\inf}-\frac{ l\chi_2}{d_3}\big)>\frac{k \chi_2}{d_3}\frac{l \chi_1}{d_3}.
 \end{equation}
 }

\noindent {\bf (H3)} {\it $a_i(t,x)$, $b_i(t,x)$, $\chi_i$ and  $d_3$, $k$ and $l$ satisfy
\begin{equation}
\label{global-existence-cond-eq02}
a_{1,inf}>\max\{0,\frac{\chi_1k(n-2)}{d_3n}\}\ ,\qquad a_{2,\inf}> \max\{0,\frac{\chi_1l(n-2)}{d_3n}\},
\end{equation}
and
\begin{equation}
\label{global-existence-cond-eq03}
b_{1,\inf}>\max\{0, \frac{\chi_2k(n-2)}{d_3n}\}\ ,\qquad  b_{2,\inf}>\max\{0,\frac{\chi_2l(n-2)}{d_3n}\}.
\end{equation}
}

 Our results  on global existence and boundedness of  nonnegative classical  solutions of \eqref{u-v-w-eq00} are stated in the following theorem.

\begin{theorem}{ (Global Existence)}
\label{thm-global-000}
\begin{itemize}
\item[(1)] Assume that { (H1)} holds. Then for any $t_0\in\RR$ and  $u_0,v_0 \in { C^+(\bar{\Omega})}$,
$\eqref{u-v-w-eq00}{ +\eqref{ic}}$ has a unique  bounded  global  classical solution $(u(x,t;t_0),v(x,t;t_0)$,$w(x,t;t_0))$  which satisfies that
\begin{equation}\label{global-existence-eq00}
\lim_{t\to  {t_0}^+ }\big(\|u(\cdot,t;t_0)-u_0\|_{ \infty}+\|v(\cdot,t;t_0)-v_0\|_{ \infty}\big)=0.
\end{equation}
 Moreover,  for any $\epsilon>0$, there is { $T(u_0,v_0,\epsilon)\ge 0$}  such that
\[0\leq u(x,t;t_0) \leq \bar A_1+\epsilon  \quad {\rm and}\quad 0\leq v(x,t;t_0)\leq \bar A_2+\epsilon  \]
for all $t\ge t_0+T(u_0,v_0,\epsilon)$, where
\begin{equation}
\label{I1-I2-overbar}
\bar A_1= \frac{a_{0,\sup}}{a_{1,\inf}-\frac{k\chi_1}{d_3}},\quad \bar A_2=\frac{b_{0,\sup}}{b_{2,\inf}-\frac{l\chi_2}{d_3}}.
\end{equation}
If $u_0\le \bar A_1+\epsilon$, $v_0\le \bar A_2+\epsilon$, then $T(u_0,v_0,\epsilon)$ can be chosen to be zero.

\item[(2)] Assume that { (H2)} holds. Then for any $t_0\in\RR$ and   $u_0,v_0 \in { C^+(\bar{\Omega})}$,
$\eqref{u-v-w-eq00}{ +\eqref{ic}}$ has a unique  bounded  global  classical solution $(u(x,t;t_0),v(x,t;t_0)$,$w(x,t;t_0))$  which satisfies \eqref{global-existence-eq00}.  Moreover,  for any $\epsilon>0$, there is { $T(u_0,v_0,\epsilon)>0$} such that
$$0\leq u(x,t;t_0) \leq \bar B_1+\epsilon
\quad {\rm and}\quad  0\leq v(x,t;t_0)\leq \bar B_2+\epsilon
$$
for all $t\ge t_0+T(u_0,v_0,\epsilon)$, where
 \begin{equation}
 \label{A1-overbar-0}\bar {B_1}=\frac{a_{0,\sup}(b_{2,\inf}-\frac{l\chi_2}{d_3})+\frac{l\chi_1}{d_3}b_{0,\sup}}{(a_{1,\inf}-\frac{k\chi_1}{d_3})(b_{2,\inf}-\frac{l\chi_2}{d_3})
 -\frac{lk\chi_1\chi_2}{d_3^2}}
 \end{equation}
and
\begin{equation}
\label{A2-overbar-0}
\bar {B_2}=\frac{b_{0,\sup}(a_{1,\inf}-\frac{k\chi_1}{d_3})+\frac{k\chi_2}{d_3}a_{0,\sup}}{(a_{1,\inf}-\frac{k\chi_1}{d_3})(b_{2,\inf}
-\frac{l\chi_2}{d_3})-\frac{lk\chi_1\chi_2}{d_3^2}}.
\end{equation}
If $u_0\le \bar B_1+\epsilon$, $v_0\le \bar B_2+\epsilon$, $T(u_0,v_0,\epsilon)$ can be chosen to be zero.

\item[(3)]Assume
 { (H3)} holds. Then for any $t_0\in\RR$ and $u_0,v_0 \in { C^+(\bar{\Omega})}$,  system  $\eqref{u-v-w-eq00}{ +\eqref{ic}}$ has a unique  bounded  global
  classical solution $(u(x,t;t_0)$, $v(x,t;t_0)$, $w(x,t;t_0))$ which satisfies \eqref{global-existence-eq00}.  Moreover,
\[0\leq\int_{\Omega}u(x,t;t_0)dx \leq \max\left\{\int_{\Omega}u_0,\frac{a_{0,\sup}}{a_{1,\inf}} \right\} \]
and
\[0\leq\int_{\Omega}v(x,t;t_0)\leq \max\left\{\int_{\Omega}v_0,\frac{b_{0,\sup}}{ b_{2,\inf}} \right\} \]
for all $t\ge t_0$.
\end{itemize}
\end{theorem}

\begin{remark}
\label{rk-1}
\begin{itemize}
\item[(1)] Under the assumption (H1), $(\bar A_1,\bar A_2)$ is the unique positive equilibrium of
the following decoupled system,
$$
\begin{cases}
u_t=u\big(a_{0,\sup}-(a_{1,\inf}-\frac{k\chi_1}{d_3}) u\big)\cr
v_t=v\big(b_{0,\sup}-(b_{2,\inf}-\frac{l\chi_2}{d_3})v\big).
\end{cases}
$$
Under the assumption (H2), $(\bar B_1,\bar B_2)$ is the unique positive equilibrium of the following cooperative system,
\begin{equation*}
\begin{cases}
u_t=u\big( a_{0,\sup}-(a_{1,\inf}-k\frac{\chi_1}{d_3}) u+l\frac{\chi_1}{d_3} v\big)\\
v_t=v\big( b_{0,\sup}-(b_{2,inf}-l\frac{\chi_2}{d_3})v+k\frac{\chi_2}{d_3}u\big).
\end{cases}
\end{equation*}

\item[(2)] Conditions {(H1)}, {(H2)} and {(H3)} are natural in the sense that when   no chemotaxis is present, i.e.,  $\chi_1=\chi_2=0,$ conditions {(H1)} and {(H2)} become the trivial conditions $a_{1,\inf}>0$ and $b_{2,\inf}>0$ while  {(H3)} becomes $a_{1,\inf}>0,$ $a_{2,\inf}>0,$ $b_{1,\inf}>0,$ and $b_{2,\inf}>0.$

\item[(3)] By {(H3)}, finite time blow up cannot happen when $n=1$ or $n=2$.  In general, it remains open whether
for any ${ t_0} \in\RR$ and $u_0,v_0\in { C^+(\bar\Omega), (u(x,t;t_0),v(x,t;t_0),w(x,t;t_0)) }$ { which satisfies \eqref{ic} } exists for all $t\ge t_0$.

\item[(4)] It is proved in \cite{ITBWS16} that,  under the assumption (H1), (H2), or (H3),  there are semitrivial entire solutions $(u^*(x,t),0,w_u^*(x,t))$ and
$(0,v^*(x,t),w_v^*(x,t))$ of \eqref{u-v-w-eq00} with
$$
\inf_{t\in\RR,x\in\bar\Omega} u^*(x,t)>0,\quad \inf_{t\in\RR,x\in\bar\Omega}v^*(x,t)>0.
$$
In the absence of chemotaxis (i.e. $\chi_1=\chi_2=0$), such semitrivial solutions are unique.

\item[(5)]
The condition of global existence and boundedness of classical solutions in \cite[Theorem 1.1(1)]{ITBRS17} implies {(H2)}. Therefore Theorem \ref{thm-global-000}(2) is an improvement of the global existence result in \cite[Theorem 1.1(1)]{ITBRS17}. { Notice also that  when $d_3=l=1,$ $a_1=\mu_1,$ $b_2=\mu_2,$ (H2) coincide with  the boundedness condition in \cite[Lemma 2.2]{STW13}. Thus {(H2)} is a generation of the global existence condition in \cite{STW13}. }
\end{itemize}
\end{remark}

We now investigate the persistence, coexistence, and extinction dynamics of \eqref{u-v-w-eq00} under the assumption (H1), (H2), or (H3).
A solution $(u(x,t),v(x,t),w(x,t))$ of \eqref{u-v-w-eq00} defined for all $t\in\RR$ is called an {\it entire solution}.
A {\it coexistence state} of \eqref{u-v-w-eq00} is an entire positive solution $(u^{**}(x,t),v^{**}(x,t),w^{**}(x,t))$ with
\begin{equation*}
\label{coexistence-eq}
\inf_{t\in\RR,x\in\bar\Omega} u^{**}(x,t)>0,\quad \inf_{t\in\RR,x\in\bar\Omega} v^{**}(x,t)>0.
\end{equation*}
We say that {\it persistence} occurs in \eqref{u-v-w-eq00} if there is $\eta>0$ such that for any $u_0,v_0\in { C^0(\bar\Omega)}$ with $u_0> 0$ and  $v_0> 0$, there is $\tau(u_0,v_0)>0$ such that
\begin{equation*}
\label{persistence-eq}
u(x,t;t_0,u_0,v_0)\ge \eta,\quad v(x,t;t_0,u_0,v_0)\ge \eta\quad \forall\,\, x\in\bar\Omega, \,\, t\ge t_0+\tau(u_0,v_0),\,\, t_0\in\RR.
\end{equation*}
We say that {\it extinction of one species}, {say,  species $u$}, occurs in \eqref{u-v-w-eq00} if for any $t_0\in\RR$ and { $u_0,v_0\in C^+(\bar\Omega)$ with $u_0\not \equiv 0$ and  $v_0\not \equiv  0$, there holds
\begin{equation*}
\label{extinction-eq1}
\lim_{t\to\infty} \|u(\cdot,t+t_0;t_0,u_0,v_0)\|_{\infty}=0, \,\, {\rm and}\,\,  \liminf_{t\to\infty}\|v(\cdot,t;t_0,u_0,v_0)\|_\infty>0.
\end{equation*}
}

To  state our results on  persistence and coexistence, we further introduce the following assumptions.

\medskip
\noindent {\bf (H4)} {\it $a_i(t,x)$, $b_i(t,x)$, $\chi_i$ and  $d_3$, $k$ and $l$ satisfy (H1) and
 \begin{equation}
\label{invariant-rectangle-cond-eq}
a_{0,\inf}>a_{2,\sup}\bar A_2\,\,\,\,\,\text{and} \,\,  \,\,\, b_{0,\inf}>b_{1,\sup}\bar A_1.
 \end{equation}}

\noindent {\bf (H5)} {\it $a_i(t,x)$, $b_i(t,x)$, $\chi_i$ and  $d_3$, $k$ and $l$ satisfy (H2) and
 \begin{equation}
\label{invariant-rectangle-cond-eq}
a_{0,inf}>\big(a_{2,\sup}-\frac{\chi_1 l}{d_3}\big)_+\bar B_2+\frac{\chi_1 l}{d_3}\bar B_2\,\,\, \,\,\text{and} \,\,\,\,\,  b_{0,inf}> \big(b_{1,\sup}
-\frac{\chi_2 k}{d_3}\big)_+\bar B_1+\frac{\chi_2 k}{d_3}\bar B_1,
 \end{equation}
 where $(\cdots)_+$ represents the positive part of the expression inside the brackets.
}

\medskip

We have the following theorem  on the persistence in  \eqref{u-v-w-eq00}.

\begin{theorem} [Persistence]
\label{thm-entire-001}
\begin{itemize}
\item[(1)]
 Assume (H4). Then there are $\underbar A_1>0$ and $\underbar A_2>0$ such that for  any $\epsilon>0$ and  $u_0,v_0\in { C^+(\bar \Omega)}$ with  $u_0,v_0\not \equiv 0$, there exists $t_{\epsilon,u_0,v_0}$ such that
\begin{equation}\label{attracting-set-eq00}
\underbar A_1 \le u(x,t;t_0,u_0,v_0) \le \bar{A_1}+\epsilon,\,\,\,\underbar A_2 \le v(x,t;t_0,u_0,v_0) \le \bar{A_2}+\epsilon
\end{equation}
for all $x\in\bar\Omega$, $t\ge t_0+t_{\epsilon,u_0,v_0}$, and $t_0\in\RR$.

\item[(2)]
 Assume (H5). Then there are $\underbar B_1>0$ and $\underbar B_2>0$ such that  for  any $\epsilon>0$  and $u_0,v_0\in { C^+(\bar \Omega)}$ with  $u_0,v_0\not \equiv 0$, there exists $t_{\epsilon,u_0,v_0}$ such \eqref{attracting-set-eq00}  holds with $\underbar A_1$,
  $\bar A_1$, $\underbar A_2$, and $\bar A_2$ being replaced by $\underbar B_1$, $\bar B_1$,   $\underbar B_2$, and  $\bar B_2$, respectively.

\end{itemize}
\end{theorem}

\begin{remark}
\label{rk-2}
\begin{itemize}
\item[(1)]  It should be pointed out that in \cite{TBJLMM16}, \cite{STW13}, \cite{TW12}, global asymptotic stability and uniqueness of coexistence states are obtained for  \eqref{u-v-eq00} when the coefficients are constants and  satisfy certain weak competition condition (see also \cite{NT13} when the system involves nonlocal terms). In such cases, the persistence follows from the asymptotic stability and uniqueness of coexistence states.
The persistence in two species chemotaxis systems without assuming the asymptotic stability of coexistence states is studied for the first time, even when the coefficients are constants.  It should be also pointed out that the authors of \cite{TaoWin} studied the persistence of a parabolic-parabolic chemotaxis system with logistic source.
The persistence in \eqref{u-v-eq00} implies the persistence of mass,  that is, if persistence occurs in \eqref{u-v-eq00}, then   for any $u_0,v_0 \in C(\bar\Omega)$ with $u_0> 0$ and  $v_0> 0$, there is $m(u_0,v_0)>0$ such that
\begin{equation*}
\int_\Omega u(x,t;t_0,u_0,v_0)dx\ge m(u_0,v_0),\quad \int_\Omega v(x,t;t_0,u_0,v_0)dx\ge m(u_0,v_0)\quad \forall\,\, t\ge t_0,\,\, t_0\in\RR.
\end{equation*}
We will study persistence in fully parabolic two species competition system with chemotaxis somewhere else.

\item[(2)] It is well known that, in the absence of  chemotaxis (i.e., $\chi_1=\chi_2=0$),
 the instability of the unique semitrivial solutions $(u^*,0)$ and $(0,v^*)$ of \eqref{u-v-eq00} implies that the persistence occurs in \eqref{u-v-eq00}. Note that both (H4) and (H5) imply
 \begin{equation}
 \label{stability-cond-1-eq1}
 a_{0,\inf} b_{2,\inf}>a_{2,\sup} b_{0,\sup},\quad b_{0,\inf} a_{1,\inf}>b_{1,\sup} a_{0,\sup},
 \end{equation}
 which implies  that  the semitrivial solutions $(u^*,0)$ and $(0,v^*)$ of \eqref{u-v-eq00} are unstable.
 When $\chi_1=\chi_2=0$,
 the conditions {(H4)} and {(H5)} coincide  and become \eqref{stability-cond-1-eq1},
  and
  $$\bar A_1=\bar B_1=\frac{a_{0,\sup}}{a_{1,\inf}},\quad \bar A_2=\bar B_2=\frac{b_{0,\sup}}{b_{2,\inf}}. $$
  Hence { Theorem} \ref{thm-entire-001} recovers  the uniform persistence result  of \eqref{u-v-eq00} in
 \cite[Theorem E(1)]{HeSh02}.

\item[(3)] The conditions {(H4)} and  {(H5)} are  sufficient conditions  for semi-trivial positive entire solutions  of \eqref{u-v-w-eq00} to be unstable.
In fact, assume (H4) or (H5) and suppose that $(u^*,0,w_u^*)$ is a semi-trivial solution of \eqref{u-v-w-eq00}. Then we have the following linearized equation of \eqref{u-v-w-eq00} at  $(u^*,0,w_u^*)$,
\begin{equation}
\label{u-v-w-lin-eq00}
\begin{cases}
u_t=d_1\Delta u-\chi_1\nabla\cdot (u^* \nabla w)-\chi_1\nabla\cdot (u \nabla w_u^*) \cr
\qquad\quad +\Big(a_0(t,x)-2a_1(t,x)u^*)u-a_2(t,x)u^*v\Big),\quad x\in \Omega\cr
v_t=d_2\Delta v-\chi_2\nabla\cdot (v \nabla w_u^*)+\Big(b_0(t,x)-b_1(t,x)u^*\Big)v,\quad x\in \Omega\cr
0=d_3\Delta w+k u+lv-\lambda w,\quad x\in \Omega \cr
\frac{\p u}{\p n}=\frac{\p v}{\p n}=\frac{\p w}{\p n}=0,\quad x\in\p\Omega.
\end{cases}
\end{equation}
Note that the second equation in \eqref{u-v-w-lin-eq00} is independent of $u$ and $w$. Assume (H4). Then
$$
u^*\le \bar A_1,\quad w_u^*\le \frac{k}{\lambda} \bar A_1
$$
 and
\begin{align*}
v_t&=d_2\Delta v-\chi_2\nabla v\cdot \nabla w_u^*-\chi_2 v \Delta w_u^* +\Big(b_0(t,x)-b_1(t,x)u^*\Big)v\\
&=d_2\Delta v-\chi_2\nabla v\cdot \nabla w_u^* +\Big(b_0(t,x)-(b_1(t,x)-\frac{\chi_2 k}{d_3})u^*-\chi_2 \frac{\lambda w_u^*}{d_3}\Big)v\\
&\ge d_2\Delta v-\chi_2\nabla v\cdot \nabla w_u^* +\Big(b_{0,\inf}-(b_{1,\sup}-\frac{\chi_2 k}{d_3})\bar A_1-\chi_2 \frac{\lambda \frac{k}{\lambda}\bar A_1}{d_3}\Big)v\\
&= d_2\Delta v-\chi_2\nabla v\cdot \nabla w_u^* +\Big(b_{0,\inf}-b_{1,\sup}\bar A_1\Big)v.
\end{align*}
This together with $b_{0,\inf}>b_{1,\sup}\bar A_1$ implies that $(u^*,0,w_u^*)$ is linearly unstable. Other cases can be proved similarly.
The proof that (H4) or (H5) implies persistence \eqref{u-v-w-eq00} is very nontrivial. To prove Theorem \ref{thm-entire-001},
we first prove five nontrivial lemmas (i.e. Lemmas \ref{persistence-lm1} to \ref{persistence-lm5}), some of which also play an important role in the study of
coexistence.

 \item[(4)] Consider the following one species parabolic-elliptic chemotaxis model,
 \begin{equation}
\begin{cases}
\label{u-w-eq00}
u_t=d_1\Delta u-\chi_1\nabla\cdot (u \nabla w)+u\Big(a_0(t,x)-a_1(t,x)u\Big),\quad x\in \Omega\cr
0=d_3\Delta w+k u-\lambda w,\quad x\in \Omega \cr
\frac{\p u}{\p n}=\frac{\p w}{\p n}=0,\quad x\in\p\Omega
\end{cases}
 \end{equation}
 and assume that
 \begin{equation}
 \label{u-w-cond}
 a_{1,\inf}>\frac{k \chi_1}{d_3}.
\end{equation}
 By the arguments of Theorem \ref{thm-entire-001}, we have  the following persistence for \eqref{u-w-eq00}, which is new.
  There is $\underbar A_1$  such that for  any $\epsilon>0,$ $t_0\in\RR,$ $u_0\in C^0(\bar \Omega)$ with $u_0\ge 0,$ and $u_0\not \equiv 0$, there exists $ t_{\epsilon,u_0}$ such that
\begin{equation}\label{attracting-set-for-u-w-eq00}
\underbar A_1 \le u(x,t;t_0,u_0) \le \bar{A_1}+\epsilon
\end{equation}
for all $x\in\bar\Omega$ and $t\ge t_0+ t_{\epsilon,u_0}$, where $(u(x,t;t_0,u_0),w(x,t;t_0,u_0))$ is the global solution of
\eqref{u-w-eq00} with $u(x,t_0;t_0,u_0)=u_0(x)$ (see Corollary \ref{u-w-cor}).
\end{itemize}
\end{remark}

The next theorem  is about the existence of coexistence states of  \eqref{u-v-w-eq00}.

\begin{theorem} [Coexistence]
\label{thm-entire-002}
\begin{itemize}
\item[(1)]
 Assume (H4). Then
there is a coexistence state $(u^{**}(x,t)$, $v^{**}(x,t)$, $ w^{**}(x,t))$ of \eqref{u-v-w-eq00}.
Moreover, the following holds.
\begin{itemize}
\item[(i)] If there is $T>0$ such that $a_i(t+T,x)=a_i(t,x),$ $b_i(t+T,x)=b_i(t,x)$ for $i=0,1,2$, then \eqref{u-v-w-eq00} has a $T$-periodic coexistence state $(u^{**}(x,t),v^{**}(x,t), w^{**}(x,t))$, that is,
 $$(u^{**}(x,t+T), v^{**}(x,t+T), w^{**}(x,t+T))=(u^{**}(x,t), v^{**}(x,t), w^{**}(x,t)).
 $$

\item[(ii)] If   $a_i(t,x)\equiv a_i(x),$ $b_i(t,x)\equiv b_i(x)$ for $i=0,1,2$, then  \eqref{u-v-w-eq00} has a steady state coexistence state
 $$(u^{**}(t,x), v^{**}(t,x),w^{**}(t,x))\equiv (u^{**}(x), v^{**}(x),w^{**}(x)).
 $$

\item[(iii)]  If  $a_i(t,x)\equiv a_i(t),$ $b_i(t,x)=b_i(t)$ for $i=0,1,2$,   then \eqref{u-v-w-eq00} has a  spatially homogeneous coexistence state
$$(u^{**}(x,t),v^{**}(x,t), w^{**}(x,t))\equiv (u^{**}(t),v^{**}(t), w^{**}(t))$$
 with $w^{**}(t)=ku^{**}(t)+lv^{**}(t)$, and if $a_i(t),$ $b_i(t)$ $(i=0,1,2)$ are periodic or almost periodic, so is   $(u^{**}(t),v^{**}(t),w^{**}(t))$.
\end{itemize}
\item[(2)]
 Assume (H5). Then there is a coexistence state $(u^{**}(x,t),v^{**}(x,t), w^{**}(x,t))$ of \eqref{u-v-w-eq00} { which satisfies (i)-(iii) of (1).}

\end{itemize}
\end{theorem}

\begin{remark}
\begin{itemize}
\item[(1)] By Theorem \ref{thm-entire-001}, (H4) or (H5) implies the persistence in  \eqref{u-v-w-eq00}.
  It is known that persistence in \eqref{u-v-eq00} implies the existence of a coexistence state. In the spatially homogeneous case,
   persistence in \eqref{u-v-w-eq00} also implies the existence of a coexistence state by the fact that the solutions of the following systems of ODEs are solutions of \eqref{u-v-w-eq00},
\begin{equation}
\begin{cases}
\label{u-v-w-ode}
u_t=u\big(a_0(t)-a_1(t)u-a_2(t)v\big)\cr
v_t=v\big(b_0(t)-b_1(t)u-b_2(t)v\big)\cr
0=k u+lv-\lambda w.
\end{cases}
 \end{equation}
 In general, it is very nontrivial to prove that persistence in \eqref{u-v-w-eq00}  implies the existence of a coexistence state.

 \item[(2)] As it is mentioned in Remark 1.2(1), when $\chi_1=\chi_2=0$,
 the conditions {(H4)} and {(H5)} coincide  and become \eqref{stability-cond-1-eq1}. Hence theorem \ref{thm-entire-002} recovers  the  coexistence result  for \eqref{u-v-eq00} in
 \cite[Theorem E(1)]{HeSh02}.

\item[(3)] {Under some additional condition, the stability and uniqueness of coexistence states are proved in our recent paper \cite{ITBWS17b}. The proofs are very nontrivial. To control the length,  we hence did not study the stability and uniqueness of coexistence states in this paper. }

\end{itemize}
\end{remark}

Our last theorem is about the extinction of { the species $u$}.

\begin{theorem}
\label{thm-extinction}
  Assume that
{ (H1)} or { (H2)},   and suppose furthermore that
\begin{equation}
\label{Asymp-exclusion-eq-00}
  b_{2,\inf} >2\frac{\chi_2 }{d_3}l,
\quad
 a_{2,\inf}\geq\frac{\chi_1 }{d_3}l,
\end{equation}
\begin{align}
\label{Asymp-exclusion-eq-03}
a_{2,\inf}\big(b_{0,\inf}(b_{2,\inf}-l\frac{\chi_2}{d_3})-b_{0,\sup}\frac{\chi_2}{d_3}l\big) \geq a_{0,\sup}\big((b_{2,\inf}-l\frac{\chi_2}{d_3})(b_{2,\sup}-l\frac{\chi_2}{d_3})-(l\frac{\chi_2}{d_3})^2\big),
\end{align}
and
\begin{align}\label{Asymp-exclusion-eq-04}
&\big(a_{1,\inf}-\frac{\chi_{1}k}{d_{3}}\big)\Big(b_{0,\inf}(b_{2,\inf}-\frac{l\chi_{2}}{d_{3}})-b_{0,\sup}\frac{l\chi_2}{d_3}\Big)\nonumber\\
&>\Big[{ \Big(\Big(b_{1,\sup}-k\frac{\chi_2}{d_3}\Big)_++k\frac{\chi_2}{d_3}\Big)}(b_{2,\inf}-\frac{\chi_{2}l}{d_{3}})+\frac{l\chi_2}{d_3}{ \Big(b_{1,\inf}-k\frac{\chi_2}{d_3}\Big)_-}\Big]a_{0,\sup}.
 \end{align}
Then for every $t_0 \in \mathbb{R}$ and $u_{0},v_0\in { C^{+}(\overline{\Omega})}$  with $\|v_0\|_\infty>0, $
 the unique bounded and globally defined classical solution $(u(\cdot,\cdot;t_0),v(\cdot,\cdot;t_0)$, $w(\cdot,\cdot;t_0))$ of  \eqref{u-v-w-eq00}{$+\eqref{ic}$}  satisfies
 \begin{equation}\label{MainAsym-eq-001}
 \lim_{t\to\infty}\left\|u(\cdot,t+t_0;t_0)\right\|_\infty=0,
 \end{equation}
\begin{equation}\label{MainAsym-eq-002}
\alpha\leq \liminf_{t \to \infty}(\min_{x \in \bar{\Omega}}v(x,t;t_0))\leq\limsup_{t \to \infty}(\max_{x \in \bar{\Omega}}v(x,t;t_0))\leq \beta,
\end{equation}
\begin{equation}\label{MainAsym-eq-003}
l\alpha\leq\lambda  \liminf_{t \to \infty}(\min_{x \in \bar{\Omega}}w(x,t;t_0))\leq \lambda\limsup_{t \to \infty}(\max_{x \in \bar{\Omega}}w(x,t,t_0))\leq l\beta,\quad \forall x \in \bar{\Omega} \quad t\geq t_0,
\end{equation}
where
$$\beta=\frac{b_{0,\sup}(b_{2,\sup}-l\frac{\chi_2}{d_3})-l\frac{\chi_2}{d_3}b_{0,\inf}}{(b_{2,\inf}-l\frac{\chi_2}{d_3})(b_{2,\sup}-l\frac{\chi_2}{d_3})-(l\frac{\chi_2}{d_3})^2},$$
and
 $$\alpha=\frac{b_{0,\inf}-l\frac{\chi_2}{d_3}\beta}{b_{2,\sup}-l\frac{\chi_2}{d_3}}>0.$$
 Furthermore, if there is a unique entire positive  solution $(v^*(x,t;\tilde b_0,\tilde b_2),w^*(x,t;\tilde b_0,\tilde b_2))$ of
 \begin{equation}
 \label{v-w-eq00}
 \begin{cases}
v_t=d_2\Delta v-\chi_2\nabla \cdot(v \nabla w)+v\Big(\tilde b_0(t,x)-\tilde b_2(t,x)v\Big),\quad x\in \Omega\cr
0=d_3\Delta w+lv-\lambda w,\quad x\in \Omega \cr
\frac{\p v}{\p n}=\frac{\p w}{\p n}=0,\quad x\in\p\Omega
\end{cases}
 \end{equation}
 for any $(\tilde b_0,\tilde b_2)\in H(b_0,b_2)$, where
 \begin{align*}
 H(b_0,b_2)=&\{(c_0(\cdot,\cdot),c_2(\cdot,\cdot))\,|\, \exists \, t_n\to\infty \,\, {\rm such\,\, that}\\
 &\lim_{n\to\infty}(b_0(t+t_n,x),b_2(t+t_n,x))=(c_0(t,x),c_2(t,x))\,\, {\rm locally \,\, uniformly\,\, in}\,\, (t,x)\in\RR\times\RR^N\},
 \end{align*}
 then
 \begin{equation}
 \label{MainAsym-eq-004}
 \lim_{t\to\infty} \| v(\cdot,t+t_0;t_0)-v^*(\cdot,t+t_0;b_0,b_2)\|_\infty=0.
 \end{equation}
\end{theorem}

\begin{remark}
\begin{itemize}
\item[(1)]
  \eqref{Asymp-exclusion-eq-03}  and \eqref{Asymp-exclusion-eq-04} imply
\begin{equation}
\label{extinction-cond-eq1}
\frac{a_{0,\sup}}{b_{0,\inf}}\le \frac{a_{2,\inf}}{b_{2,\sup}},\quad \frac{a_{0,\sup}}{b_{0,\inf}}< \frac{a_{1,\inf}}{b_{1,\sup}}.
\end{equation}
{To see this, we first note that \eqref{Asymp-exclusion-eq-04}  implies that
 \begin{align*}
\big(a_{1,\inf}-\frac{\chi_{1}k}{d_{3}}\big)b_{0,\inf}(b_{2,\inf}-\frac{l\chi_{2}}{d_{3}})
&>\Big[{ \Big(\Big(b_{1,\sup}-k\frac{\chi_2}{d_3}\Big)_++k\frac{\chi_2}{d_3}\Big)}(b_{2,\inf}-\frac{\chi_{2}l}{d_{3}})\Big]a_{0,\sup}\nonumber\\
&\geq b_{1,\sup}(b_{2,\inf}-\frac{\chi_{2}l}{d_{3}})a_{0,\sup}.
 \end{align*}
 Thus since $b_{2,\inf}-\frac{\chi_{2}l}{d_{3}}>0,$ we get
 $$\big(a_{1,\inf}-\frac{\chi_{1}k}{d_{3}}\big)b_{0,\inf}>b_{1,\sup}a_{0,\sup},$$ which implies
 the second inequality in \eqref{extinction-cond-eq1}.
 Second, note that \eqref{Asymp-exclusion-eq-03} implies that
 \begin{align*}
a_{2,\inf}\big(b_{0,\inf}(b_{2,\inf}-l\frac{\chi_2}{d_3})-b_{0,\sup}\frac{\chi_2}{d_3}l\big)&\geq a_{0,\sup}\big((b_{2,\inf}-l\frac{\chi_2}{d_3})b_{2,\sup}-l\frac{\chi_2}{d_3}b_{2,\inf}\big)\nonumber\\
&\geq a_{0,\sup}(b_{2,\inf}-2l\frac{\chi_2}{d_3})b_{2,\sup}.
\end{align*}
%
This together with the fact that $a_{2,\inf}\big(b_{0,\inf}(b_{2,\inf}-l\frac{\chi_2}{d_3})-b_{0,\sup}\frac{\chi_2}{d_3}l\big)\leq a_{2,\inf}b_{0,\inf}(b_{2,\inf}-2l\frac{\chi_2}{d_3}) $  implies that
$$
a_{2,\inf}b_{0,\inf}(b_{2,\inf}-2l\frac{\chi_2}{d_3}) \geq a_{0,\sup}(b_{2,\inf}-2l\frac{\chi_2}{d_3})b_{2,\sup},
$$
which combines with $b_{2,\inf}-2l\frac{\chi_2}{d_3}>0$ implies the first inequality in \eqref{extinction-cond-eq1}.}

\item[(2)]
 When $\chi_1=\chi_2=0$, \eqref{Asymp-exclusion-eq-00} becomes
$$
b_{2,\inf}>0,\quad a_{2,\inf}>0,\quad a_{1,\inf}>0;
$$
\eqref{Asymp-exclusion-eq-03} and  \eqref{Asymp-exclusion-eq-04} become
$$
\frac{a_{0,\sup}}{b_{0,\inf}}\le \frac{a_{2,\inf}}{b_{2,\sup}}\quad {\rm and}\quad
\frac{a_{0,\sup}}{b_{0,\inf}}< \frac{a_{1,\inf}}{b_{1,\sup}},
$$
respectively.
Therefore, the extinction results for \eqref{u-v-eq00} in \cite{HeSh02} are recovered.

\item[(3)] When the coefficients are constants, Theorem \ref{thm-extinction} coincide with the exclusion Theorem in  \cite[Theorem 1.4]{ITBRS17}. Thus Theorem \ref{thm-extinction} give a  natural extension to the  phenomenon of exclusion  in heterogeneous media.

    \item[(4)] The reader is referred to \cite{ITBWS16} for the existence and uniqueness of positive entire solutions of \eqref{v-w-eq00}.
\end{itemize}

\end{remark}

{The results established in this paper provide  various conditions  under which
 persistence or extinction occurs.   All the conditions depend on the chemotaxis sensitivity coefficients $\chi_1$ and $\chi_2$, which reflect some effects  of chemotaxis on the persistence and extinction of the system.
 However,  it remains open whether  chemotaxis makes species easier to persist or go extinct. This is a very interesting issue and we plan to study  it somewhere else.
}

The rest of the paper is organized as follows. In section 2, we study the global existence of classical solutions and prove Theorem \ref{thm-global-000}. Section 3 is devoted to the study of the persistence and boundedness  of  classical  solutions. It is here that we present the proof of Theorem \ref{thm-entire-001}.
In section 4, we study the existence of coexistence states and prove Theorem \ref{thm-entire-002}.
Finally in section 5, we study the phenomenon of exclusion and prove Theorem \ref{thm-extinction}.

\medskip

\section{Global existence of bounded classical solutions}
In this section, we study the existence of bounded classical solutions of system \eqref{u-v-w-eq00} and prove Theorem \ref{thm-global-000}. We start with the following important result on the local existence of classical solutions of system \eqref{u-v-w-eq00} with  initial  functions in ${ C^+(\bar{\Omega}).}$
\begin{lemma}
\label{lm-local-001}
 For any given $t_0 \in \mathbb{R},$ $u_0,v_0 \in { C^+(\bar{\Omega})}$,  there exists $T_{\max}(t_0,u_0,v_0) \in (0,\infty]$  such that $\eqref{u-v-w-eq00}{ +\eqref{ic}}$  has a unique nonnegative classical solution
$(u(x,t;t_0)$, $v(x,t;t_0)$, $w(x,t;t_0))$ on $(t_0,t_0+T_{\max}(t_0,u_0,v_0))$ satisfying that
$$\lim_{t \nearrow t_0}\|u(\cdot,t;t_0)-u_0\|_{\infty}=0,\quad
\lim_{t\nearrow
t_0}\|v(\cdot,t;t_0)-v_0\|_{ \infty}=0, $$
 and
moreover if $T_{\max}(t_0,u_0,v_0)< \infty,$ then
\begin{equation}
\label{local-eq00}
\limsup_{t \nearrow T_{\max}(t_0,u_0,v_0)}\left( \left\| u(\cdot,t_0+t;t_0) \right\|_{\infty} +\left\| v(\cdot,t_0+t;t_0) \right\|_{ \infty}\right) =\infty.
\end{equation}
\end{lemma}
\begin{proof}
If follows from the similar arguments as those in \cite[Lemma 2.1]{STW13}.
\end{proof}

Next, we consider the following system of ODEs induced from system \eqref{u-v-w-eq00},
\begin{equation}
\label{ode00}
\begin{cases}
\overline{u}'=\frac{\chi_1}{d_3} \overline{u}\big(k \overline {u}+l\overline v-k\underline{u}-l\underline{v}\big)+ \overline{u}\big[a_{0,\sup}(t)-a_{1,\inf}(t)\overline u
-a_{2,\inf}(t)\underline{v}\big]\\
\underline{u}'=\frac{\chi_1}{d_3} \underline{u}\big(k \underline {u}+l\underline v-k\overline{u}-l\overline{v}\big)+ \underline{u}\big[a_{0,\inf}(t)-a_{1,\sup}(t)\underline u
-a_{2,\sup}(t)\overline{v}\big]\\
\overline{v}'=\frac{\chi_2}{d_3} \overline{v}\big(k \overline {u}+l\overline v-k\underline{u}-l\underline{v}\big)+ \overline{v}\big[b_{0,\sup}(t)-b_{2,\inf}(t)\overline v
-b_{1,\inf}(t)\underline{u}\big]\\
\underline{v}'=\frac{\chi_2}{d_3} \underline{v}\big(k \underline {u}+l\underline v-k\overline{u}-l\overline{v}\big)+ \underline{v}\big[b_{0,\inf}(t)-b_{2,\sup}(t)\underline v
-b_{1,\sup}(t)\overline{u}\big].
\end{cases}
\end{equation}
For convenience, we let
\begin{align*}
&\left(\overline{u}(t),\underline{u}(t),\overline{v}(t),\underline{v}(t)\right)\\
&=\left(\overline{u}\left(t;t_0,\overline{u}_0,\underline{u}_0,\overline{v}_0,\underline{v}_0\right),\underline{u}\left(t;t_0,\overline{u}_0,\underline{u}_0,\overline{v}_0,\underline{v}_0\right),\overline{v}\left(t;t_0,\overline{u}_0,\underline{u}_0,\overline{v}_0,\underline{v}_0\right),\underline{v}\left(t;t_0,\overline{u}_0,\underline{u}_0,\overline{v}_0,
\underline{v}_0\right)\right)
\end{align*}
be the solution of \eqref{ode00}
with initial condition
\begin{align}\label{initial-ode00}
&\left(\overline{u}\left(t_0;t_0,\overline{u}_0,\underline{u}_0,\overline{v}_0,\underline{v}_0\right),\underline{u}\left(t_0;t_0,\overline{u}_0,
\underline{u}_0,\overline{v}_0,\underline{v}_0\right),\overline{v}\left(t_0;t_0,\overline{u}_0,\underline{u}_0,\overline{v}_0,\underline{v}_0\right),
\underline{v}\left(t_0;t_0,\overline{u}_0,\underline{u}_0,\overline{v}_0,
\underline{v}_0\right)\right)\nonumber\\
&=\left(\overline{u}_0,\underline{u}_0,\overline{v}_0,\underline{v}_0\right) \in \mathbb{R}^4_+.
\end{align}
Then for given $t_0\in\mathbb{R}$ and  $\left(\overline{u}_0,\underline{u}_0,\overline{v}_0,\underline{v}_0\right) \in \mathbb{R}^4_+,$ there exists $T_{\max}\left(t_0,\overline{u}_0,\underline{u}_0,\overline{v}_0,\underline{v}_0\right)>0$ such that \eqref{ode00} has a unique classical solution
$\left(\overline{u}(t),\underline{u}(t),\overline{v}(t),\underline{v}(t)\right)$ on $(t_0,t_0+T_{\max}\left(t_0,\overline{u}_0,\underline{u}_0,\overline{v}_0,\underline{v}_0\right))$ satisfying \eqref{initial-ode00}. Moreover if
$T_{\max}\left(t_0,\overline{u}_0,\underline{u}_0,\overline{v}_0,\underline{v}_0\right)<\infty,$ then
\begin{equation}\label{blow-creterion-ode00}
\limsup_{t \nearrow T_{\max}\left(t_0,\overline{u}_0,\underline{u}_0,\overline{v}_0,\underline{v}_0\right)}
\left(|\overline{u}(t_0+t)|+|\underline{u}(t_0+t)|+|\overline{v}(t_0+t)|+|\underline{v}(t_0+t)|\right)=\infty.
\end{equation}
We now state and prove the following important lemma which provides  sufficient conditions for the boundedness of classical solutions of system \eqref{ode00}.

\begin{lemma}\label{lem-1-ode00} { Let} $\left(\overline{u}(t),\underline{u}(t),\overline{v}(t),\underline{v}(t)\right)$ be the solution of \eqref{ode00} which satisfies \eqref{initial-ode00}. Then
\begin{itemize}
\item[(i)] $0\leq\underline{u}_0 \leq \overline{u}_0\quad \text{and} \quad 0\leq \underline{v}_0 \leq \overline{v}_0$ imply
 $ 0\leq\underline{u}(t )\leq \overline{u}(t) \quad \text{and} \quad 0\leq \underline{v}(t)\leq \overline{v}(t)$ for all $t \in [t_0,t_0+T_{\max}\left(\overline{u}_0,\underline{u}_0,\overline{v}_0,\underline{v}_0\right)).$

\item[(ii)] If { (H2)} holds, then $T_{\max}\left(t_0,\overline{u}_0,\underline{u}_0,\overline{v}_0,\underline{v}_0\right)=\infty$ and
 $$
\limsup_{t\to\infty} \overline{u}(t) \leq \bar B_1,\quad \limsup_{t\to\infty}\overline{v}(t)\leq  \bar B_2,
$$
 where $\bar B_1$ and $\bar B_2$ are as in \eqref{A1-overbar-0} and \eqref{A2-overbar-0}, respectively.
\end{itemize}
\end{lemma}

\begin{proof}
(i) Let $\epsilon>0$ and $\left(\overline{u}_{\epsilon}(t),\underline{u}_{\epsilon}(t),\overline{v}_{\epsilon}(t),\underline{v}_{\epsilon}(t) \right)$ { be} the solution of \eqref{ode00} { with $a_{0,\sup}(t)$ and $b_{0,\sup}(t)$ being replaced by $a_{0,\sup}(t)+\epsilon$ and $b_{0,\sup}(t)+\epsilon$, respectively,
and}  satisfying \eqref{initial-ode00} with  $\overline{u}_0,$ $\overline{v}_0$ being replaced  respectively by $\overline{u}^{\epsilon}_0=\overline{u}_0+\epsilon$ and $\overline{v}^{\epsilon}_0=\overline{v}_0+\epsilon.$
 We claim first that (i) holds for $\left(\overline{u}_{\epsilon}(t),\underline{u}_{\epsilon}(t),\overline{v}_{\epsilon}(t),\underline{v}_{\epsilon}(t) \right).$ Suppose by contradiction that our claim does not hold. Then there exists $\overline{t} \in (t_0,t_0+T_{\max}\left(t_0,\overline{u}^{\epsilon}_0,\underline{u}_0,\overline{v}^{\epsilon}_0,\underline{v}_0\right))$ such that
\begin{equation}\label{proof-eq00-ode00}
0\leq\underline{u}_{\epsilon}(t ) < \overline{u}_{\epsilon}(t),\quad 0\leq\underline{v}_{\epsilon}(t)< \overline{v}_{\epsilon}(t) \, ,\forall t \in [t_0,\overline{t})
\end{equation}
and
{
$$
{\rm either}\quad \underline{u}_{\epsilon}(\overline{t} ) = \overline{u}_{\epsilon}(\overline{t})\quad {\rm or}\quad
\underline{v}_{\epsilon}(\overline{t} ) = \overline{v}_{\epsilon}(\overline{t}).
$$
Without loss of generality, assume that $\underline{u}_{\epsilon}(\overline{t} ) = \overline{u}_{\epsilon}(\overline{t})$.
Then}
 on one hand \eqref{proof-eq00-ode00} implies that
\begin{equation}\label{proof-eq01-ode00}
\left(\overline{u}_{\epsilon}-\underline{u}_{\epsilon}\right)^{'}(\overline{t})\leq 0,
\end{equation}
and on the other hand the difference between the first and the second equations of \eqref{ode00} gives
\begin{align*}
\left(\overline{u}_{\epsilon}-\underline{u}_{\epsilon}\right)^{'}(\overline{t})&=\overline{u}_{\epsilon}(\overline{t})\left\{a_{0,\sup}(\overline{t})+\epsilon-a_{0,\inf}(\overline{t})+(a_{1,\sup}(\overline{t})-a_{1,\inf}(\overline{t}))\overline{u}_{\epsilon}(\overline{t})+2l\frac{\chi_1}{d_3}\left(\overline{v}_{\epsilon}-\underline{v}_{\epsilon}\right)(\overline{t})\right\}\nonumber\\
&\qquad \qquad \qquad +\overline{u}_{\epsilon}(\overline{t})\left\{a_{2,\sup}(\overline{t})\overline{v}_{\epsilon}(\overline{t})-a_{2,\inf}(\overline{t})
\underline{v}_{\epsilon}(\overline{t})\right\}>0,
\end{align*}
which contradicts to \eqref{proof-eq01-ode00}.  Thus (i) holds for $\big(\overline{u}_{\epsilon}(t),\underline{u}_{\epsilon}(t)$, $\overline{v}_{\epsilon}(t)$, $\underline{v}_{\epsilon}(t) \big)$. Letting $\epsilon \to 0$, we have that (i) holds for $\left(\overline{u}(t),\underline{u}(t),\overline{v}(t),\underline{v}(t) \right)$.

(ii) First from the first and third equations of \eqref{ode00} we get
\begin{equation}
\label{ode01}
\begin{cases}
\overline{u}' \leq \overline{u}\Big[a_{0,\sup}-\left(a_{1,\inf}-k\frac{\chi_1}{d_3}\right)\overline u+l\frac{\chi_1}{d_3}\overline{v}\Big]\\
\overline{v}' \leq \overline{v}\Big[b_{0,\sup}-\left(b_{2,\inf}-l\frac{\chi_2}{d_3}\right)\overline v+k\frac{\chi_2}{d_3}\overline{u}\Big].
\end{cases}
\end{equation}
Thus the result follows from  comparison principle for cooperative systems
  and the fact that   $(\bar B_1, \bar B_2)$ is a uniformly asymptotically stable solution for the following system of ODEs,
\begin{equation*}
\begin{cases}
u'=u\left\{ a_{0,\sup}-(a_{1,\inf}-k\frac{\chi_1}{d_3}) u+l\frac{\chi_1}{d_3} v\right\}\\
v'=v\left\{ b_{0,\sup}-\left(b_{2,inf}-l\frac{\chi_2}{d_3}\right)v+k\frac{\chi_2}{d_3}u\right\}.
\end{cases}
\end{equation*}
\end{proof}

Now we prove Theorem \ref{thm-global-000}.

\begin{proof} [{ Proof of Theorem \ref{thm-global-000}}]

Let  $u_0,v_0 \in {C^+(\bar{\Omega})}.$

(1) From the first equation of system \eqref{u-v-w-eq00}, we have that for $t\in (t_0,t_0+T_{\max}(t_0,u_0,v_0))$,
\begin{align*}
u_t=&d_1\Delta u-\chi_1 \nabla u \cdot \nabla w+u\left\{ a_0(t,x)-\left(a_1(t,x)-k\frac{\chi_1 }{d_3}\right)u -\left(a_2(t,x)-l\frac{\chi_1 }{d_3} \right)v-\frac{\chi_1 }{d_3}\lambda w\right\}\\
& \leq d_1\Delta u-\chi_1 \nabla u \cdot \nabla w+u\left\{ a_{0,\sup}-\left(a_{1,\inf}-k\frac{\chi_1 }{d_3}k\right)u-\left(a_{2,\inf}-l\frac{\chi_1 }{d_3}\right) v - \frac{\chi_1 }{d_3}\lambda w \right\}.
\end{align*}
This together with { (H1)} gives for $t\in (t_0,t_0+T_{\max}(t_0,u_0,v_0))$,
\begin{align}
\label{aux-eq1}
u_t& \leq d_1\Delta u-\chi_1 \nabla u \cdot \nabla w+u\left\{ a_{0,\sup}-\left(a_{1,\inf}-k\frac{\chi_1 }{d_3}k\right)u \right\}.
\end{align}
{ Let $u(t;\|u_0\|_\infty)$ be the solution of the ODE
$$
u^{'}=u\left\{ a_{0,\sup}-\left(a_{1,\inf}-k\frac{\chi_1 }{d_3}k\right)u \right\}
$$
with $u(0;\|u_0\|_\infty)=\|u_0\|_\infty$.
Then $u(t;\|u_0\|_\infty)$ is increasing if $\|u_0|_\infty<\frac{a_{0,\sup}}{a_{1,\inf}-k\frac{\chi_1}{d_3}}$ and is decreasing if
$\|u_0\|_\infty>\frac{a_{0,\sup}}{a_{1,\inf}-k\frac{\chi_1}{d_3}}$,  and $u(t;\|u_0\|_\infty)$ converges to $\frac{a_{0,\sup}}{a_{1,\inf}-k\frac{\chi_1}{d_3}}$ as $t\to\infty$. }
Therefore by  comparison principle  for parabolic equations, we get
 \begin{equation}
 \label{global-u-eq}
 0\leq u(x,t;t_0,u_0,v_0)\leq \max\left\{\|u_0\|_\infty,\frac{a_{0,\sup}}{a_{1,\inf}-k\frac{\chi_1}{d_3}} \right\}\quad \forall\,\, t\in [t_0,t_0+T_{\max}(t_0,u_0,v_0)).
 \end{equation}

Similarly, the second equation of system \eqref{u-v-w-eq00} gives
  \begin{equation}
  \label{global-v-eq}
  0\leq v(x,t;t_0,u_0,v_0)\leq \max\left\{\|v_0\|_\infty,\frac{b_{0,\sup}}{b_{2,\inf}-l\frac{\chi_2}{d_3}} \right\}\quad \forall\,\, t\in [t_0,t_0+T_{\max}(t_0,u_0,v_0)).
  \end{equation}
By \eqref{local-eq00}, \eqref{global-u-eq}, and \eqref{global-v-eq}, we have $T_{\max}(t_0,u_0,v_0)=\infty$.

Moreover, by \eqref{aux-eq1} and
comparison principle for parabolic equations again,
{ for any $\epsilon>0$, there is $T_1(u_0,v_0,\epsilon)\ge 0$ such that
$$
0\le u(x,t;t_0,u_0,v_0)\le  \frac{a_{0,\sup}}{a_{1,\inf}-k\frac{\chi_1}{d_3}}+\epsilon\quad \forall\,\, x\in\bar\Omega,\,\, t\ge t_0+T_1(u_0,v_0,\epsilon),
$$
and $T_1(u_0,v_0,\epsilon)$ can be chosen to be zero if $u_0\le\bar A_1+\epsilon$.
Similarly, for any $\epsilon>0$, there is $T_2(u_0,v_0,\epsilon)\ge 0$ such that
$$
0\le v(x,t;t_0,u_0,v_0)\le \frac{b_{0,\sup}}{b_{2,\inf}-l\frac{\chi_2}{d_3}}+\epsilon \quad \forall\,\, x\in\bar\Omega,\,\, t\ge t_0+T_2(u_0,v_0,\epsilon),
$$
and $T_2(u_0,v_0,\epsilon)$ can be chosen to be zero if $v_0\le\bar A_2+\epsilon$.
(1) thus follows with $T(u_0,v_0,\epsilon)=\max\{T_1(u_0,v_0,\epsilon),T_2(u_0,v_0,\epsilon)\}$.}

(2) Let $\overline{u}_0=\max_{x \in \bar \Omega}u_0(x),$ $\underline{u}_0=\min_{x \in \bar \Omega}u_0(x),$ $\overline{v}_0=\max_{x \in \bar \Omega}v_0(x)$ , $\underline{v}_0=\min_{x \in \bar \Omega}v_0(x)$ and let $\left(\overline{u}(t),\underline{u}(t),\overline{v}(t),\underline{v}(t) \right)$ { be} the solution of \eqref{ode00} satisfying initial condition \eqref{initial-ode00}.
 By the similar arguments as those in \cite[Theorem 1.1(1)]{ITBRS17}, under the condition { (H2)} we have
\[
0\leq \underline{u}(t) \leq u(x,t) \leq \overline{u}(t) \quad \text{and} \quad  0 \leq \underline{v}(t) \leq v(x,t) \leq \overline{v}(t) \, , \forall  x \in \bar \Omega \, \,\,  t \in (t_0, t_0+T_{\max}).
\]
This together with Lemma \ref{lem-1-ode00} implies Theorem \ref{thm-global-000} (2).

(3) It follows from the similar arguments as those in \cite[Theorem 1.1(2)]{ITBRS17}.
\end{proof}

\section{Persistence}
\label{Persistence}

In this section, we study the persistence  in \eqref{u-v-w-eq00} and prove Theorem \ref{thm-entire-001}.

Fix $T>0$. We first prove five Lemmas.

\begin{lemma}
\label{persistence-lm1}
\begin{itemize}
\item[(1)] Assume (H1).
For any $\epsilon>0$, there is $\delta=\delta(\epsilon)>0$ such that for any { $u_0, v_0 \in C^+(\bar{\Omega})$, the solution $(u(x,t;t_0)$, $v(x,t;t_0)$, $w(x,t;t_0))$  of $\eqref{u-v-w-eq00}+\eqref{ic}$ satisfies the following.}
\begin{itemize}
\item[(i)] If $0\le u_0\le \delta$, then $u(x,t;t_0)\le \epsilon$ for $t\in [t_0,t_0+T]$ and $x\in\bar\Omega$.

\item[(ii)] If $0\le v_0\le \delta$, then $v(x,t;t_0)\le \epsilon$ for $t\in [t_0,t_0+T]$ and $x\in\bar\Omega$.
\end{itemize}

\item[(2)] Assume (H2).
For any $\epsilon>0$, there is $\delta=\delta(\epsilon)>0$ such that for any  { $u_0, v_0 \in C^+(\bar{\Omega})$,}
 { the solution $(u(x,t;t_0)$, $v(x,t;t_0)$, $w(x,t;t_0))$  of $\eqref{u-v-w-eq00} +\eqref{ic}$ satisfies the following.}
\begin{itemize}
\item[(i)] If $0\le u_0\le \delta$ { and  $0\le v_0\le \bar B_2+\epsilon$}, then $u(x,t;t_0)\le \epsilon$ for $t\in [t_0,t_0+T]$ and $x\in\bar\Omega$.

\item[(ii)] If $0\le v_0\le \delta$ { and $0\le u_0\le \bar B_1+\epsilon$}, then $v(x,t;t_0)\le \epsilon$ for $t\in [t_0,t_0+T]$ and $x\in\bar\Omega$.
\end{itemize}
\end{itemize}
\end{lemma}

\begin{proof}
(1)(i)
Assume (H1). Then
\begin{align*}
u_t&=d_1\Delta u-\chi_1 \nabla \cdot (u\nabla w)+u\Big(a_0(t,x)-a_1(t,x)u-a_2(t,x)v\Big)\\
&=d_1\Delta u-\chi_1 \nabla u\cdot\nabla w+u\Big(a_0(t,x)-(a_1(t,x)-\frac{\chi_1 k}{d_3})u-(a_2(t,x)-\frac{\chi_1 l}{d_3})v-\frac{\chi_1 \lambda}{d_3}w\Big)\\
&\le d_1\Delta u-\chi_1 \nabla u\cdot\nabla w+a_{0,\sup} u.
\end{align*}
Hence, by comparison principle for parabolic equations,
we have
$$
u(x,t;t_0)\le e^{a_{0,\sup}(t-t_0)} \|u_0\|\quad \forall\,\, t\ge t_0.
$$
(1)(i) thus follows with $\delta=\epsilon e^{-a_{0,\sup}T}$ for any given $\epsilon>0$.

(1)(ii) It can be proved by the similar arguments as in (1)(i).

(2)(i) By Theorem \ref{thm-global-000}(2),
$$
 v(x,t+t_0;t_0)\le \bar B_2+\epsilon
\quad \forall \,\, t\ge { t_0},\quad x\in\bar\Omega.
$$
Assume (H2). Then
\begin{align*}
u_t&=d_1\Delta u-\chi_1 \nabla \cdot (u\nabla w)+u\Big(a_0(t,x)-a_1(t,x)u-a_2(t,x)v\Big)\\
&=d_1\Delta u-\chi_1 \nabla u\cdot\nabla w+u\Big(a_0(t,x)-(a_1(t,x)-\frac{\chi_1 k}{d_3})u-(a_2(t,x)-\frac{\chi_1 l}{d_3})v-\frac{\chi_1 \lambda}{d_3}w\Big)\\
&\le d_1\Delta u-\chi_1 \nabla u\cdot\nabla w+\big(a_{0,\sup}+\frac{\chi_1 l}{d_3}(\bar B_2+\epsilon)\big) u.
\end{align*}
By comparison principle for parabolic equations,
we have
$$
u(x,t;t_0)\le e^{\big(a_{0,\sup}+\frac{\chi_1 l}{d_3}(\bar B_2+\epsilon)\big)(t-t_0)} \|u_0\|\quad \forall\,\, t\ge t_0.
$$
(2)(i) thus follows with $\delta=\epsilon e^{-\big(a_{0,\sup}+\frac{\chi_1 l}{d_3}(\bar B_2+\epsilon)\big)T}$ for any given $\epsilon>0$.

(2)(ii) It can be proved by the similar arguments as in (2)(i).
\end{proof}

\begin{lemma}
\label{persistence-lm2}
\begin{itemize}
\item[(1)]
{Assume  (H4). Let $\epsilon_0$ and $\delta_0=\delta_0(\epsilon_0)$ be such that Lemma
\ref{persistence-lm1}(1) holds with $\epsilon=\epsilon_0$ and $\delta=\delta_0$,
$$
a_{0,\inf}>a_{2,\sup}(\bar A_2+\epsilon_0)+\frac{\chi_1 k}{d_3}\epsilon_0,\quad
b_{0,\inf}> b_{1,\sup}(\bar A_1+\epsilon_0)+\frac{\chi_2 l}{d_3}\epsilon_0,
$$
and
\begin{align*}
\delta_0<{ \min}\Big\{\frac{a_{0,\inf}-a_{2,\sup}(\bar A_2+\epsilon_0)-\frac{\chi_1 k}{d_3}\epsilon_0}{a_{1,\sup}-\frac{\chi_1 k}{d_3}},\frac{b_{0,\inf}- b_{1,\sup}(\bar A_1+\epsilon_0)-\frac{\chi_2 l}{d_3}\epsilon_0}{b_{2,\sup}-\frac{\chi_2 l}{d_3}}\Big\}.
\end{align*}
} For given { $u_0, v_0 \in C^+(\bar{\Omega})$, the solution $(u(x,t;t_0)$, $v(x,t;t_0)$, $w(x,t;t_0))$  of $\eqref{u-v-w-eq00} +\eqref{ic}$ satisfies the following.}
\begin{itemize}
\item[(i)] If $0<u_0<\delta_0$ and  { $0\le v_0\le \bar A_2+\epsilon$}, then
$u(x,t+t_0;t_0)>\inf_{ \bar{\Omega}} u_0(x)\quad \forall \,\, 0<t\le T.
$

\item[(ii)] If $0<v_0<\delta_0$ and {  $0\le u_0\le  \bar A_1+\epsilon$}, then
$
v(x,t+t_0;t_0)>\inf_{ \bar{\Omega}} v_0(x)\quad \forall \,\, 0<t\le T.
$
\end{itemize}

\item[(2)]
{ Assume  (H5). Let $\epsilon_0$ and $\delta_0=\delta_0(\epsilon_0)$  be such that Lemma
\ref{persistence-lm1}(2) holds with $\epsilon=\epsilon_0$ and $\delta=\delta_0$,
$$
a_{0,\inf}>\big[\big(a_{2,\sup}-\frac{\chi_1 l}{d_3}\big)_++\frac{\chi_1 l}{d_3}\big](\bar B_2+\epsilon_0)+\frac{\chi_1 k}{d_3}\epsilon_0,
$$
$$
b_{0,\inf}>\big[\big(b_{1,\sup}-\frac{\chi_2 k}{d_3}\big)_++\frac{\chi_2 k}{d_3}\big](\bar B_1+\epsilon_0)+\frac{\chi_2 l}{d_3}\epsilon_0,
$$
and
\begin{align*}
\delta_0<{ \min}\Big\{&
\frac{a_{0,\inf}-\big[\big(a_{2,\sup}-\frac{\chi_1 l}{d_3}\big)_++\frac{\chi_1 l}{d_3}\big](\bar B_2+\epsilon_0)-\frac{\chi_1 k}{d_3}\epsilon_0}
{a_{1,\sup}-\frac{\chi_1 k}{d_3}},\\
& \frac{b_{0,\inf}-\big[\big(b_{1,\sup}-\frac{\chi_2 k}{d_3}\big)_++\frac{\chi_2 k}{d_3}\big](\bar B_1+\epsilon_0)-\frac{\chi_2 l}{d_3}\epsilon_0}
{b_{2,\sup}-\frac{\chi_2 l}{d_3}}\Big\}.
\end{align*}
} For given  { $u_0, v_0 \in C^+(\bar{\Omega})$,}
{ the solution $(u(x,t;t_0)$, $v(x,t;t_0)$, $w(x,t;t_0))$  of $\eqref{u-v-w-eq00} +\eqref{ic}$ satisfies the following.}
\begin{itemize}
\item[(i)] If $0<u_0<\delta_0$ and {  $0\le v_0\le \bar B_2+\epsilon$}, then
$u(x,t+t_0;t_0)>\inf_{ \bar{\Omega}} u_0(x)\quad \forall \,\, 0<t\le T.$

\item[(ii)] If $0<v_0<\delta_0$ and {  $0\le u_0\le \bar B_1+\epsilon$}, then
$v(x,t+t_0;t_0,)>\inf_{ \bar{\Omega}} v_0(x)\quad \forall \,\, 0<t\le T.$
\end{itemize}
\end{itemize}
\end{lemma}

\begin{proof}
(1)(i) Without loss of generality, assume $\inf_{ \bar{\Omega}}u_0(x)>0$. By Theorem \ref{thm-global-000} (1),
$$
v(x,t+t_0;t_0)\le \bar A_2+\epsilon_0
\quad \forall \,\, t\ge { t_0},\quad x\in\bar\Omega.
$$
This together with Lemma \ref{persistence-lm1} (1) implies that
\begin{align*}
u_t&=d_1\Delta u-\chi_1 \nabla \cdot (u\nabla w)+u\Big(a_0(t,x)-a_1(t,x)u-a_2(t,x)v\Big)\\
&=d_1\Delta u-\chi_1 \nabla u\cdot\nabla w+u\Big(a_0(t,x)-(a_1(t,x)-\frac{\chi_1 k}{d_3})u-(a_2(t,x)-\frac{\chi_1 l}{d_3})v-\frac{\chi_1 \lambda}{d_3}w\Big)\\
&\ge d_1\Delta u-\chi_1 \nabla u\cdot\nabla w\\
&\,\, +u\Big(a_0(t,x)-(a_1(t,x)-\frac{\chi_1 k}{d_3})u-(a_2(t,x)-\frac{\chi_1 l}{d_3})(\bar A_2+\epsilon_0)-\frac{\chi_1 \lambda}{d_3}\big(\frac{k}{\lambda} \epsilon_0+\frac{l}{\lambda} (\bar A_2+\epsilon_0)\big)\Big) \\
&=d_1\Delta u-\chi_1 \nabla u\cdot\nabla w+u\Big(a_0(t,x)-a_2(t,x)(\bar A_2+\epsilon_0)-\frac{\chi_1 k}{d_3}\epsilon_0-(a_1(t,x)-\frac{\chi_1 k}{d_3})u\Big)\\
&\ge d_1\Delta u-\chi_1 \nabla u\cdot\nabla w+u\Big(a_{0,\inf}-a_{2,\sup}(\bar A_2+\epsilon_0)-\frac{\chi_1 k}{d_3}\epsilon_0-(a_{1,\sup}-\frac{\chi_1 k}{d_3})u\Big)
\end{align*}
for $0<t\le T$. Let $\tilde u(t)$ be the solution of
$$
\tilde u_t=\tilde u\Big(a_{0,\inf}-a_{2,\sup}(\bar A_2+\epsilon_0)-\frac{\chi_1 k}{d_3}\epsilon_0-(a_{1,\sup}-\frac{\chi_1 k}{d_3})\tilde u\Big)
$$
with $\tilde u(t_0)=\inf_{ \bar\Omega} u_0(x)$.  We have
$\tilde u(t)$ is monotonically increasing in $t\ge t_0$ and
$$\lim_{t\to\infty} \tilde u(t)=\frac{a_{0,\inf}-a_{2,\sup}(\bar A_2+\epsilon_0)-\frac{\chi_1 k}{d_3}\epsilon_0}{(a_{1,\sup}-\frac{\chi_1 k}{d_3})}.
$$
By comparison principle for parabolic equations, we have
$$
u(x,t+t_0;t_0)\ge \tilde u(t+t_0)> \inf_{ \bar\Omega} u_0(x)\quad \forall\,\, 0<t\le T.
$$

(1)(ii) It can be proved by the similar arguments as those in (1)(i).

(2)(i) (i) Again, without loss of generality, assume $\inf_{ \bar\Omega}u_0(x)>0$. By Theorem \ref{thm-global-000} (2),
$$
v(x,t+t_0;t_0)\le \bar B_2+\epsilon_0
\quad \forall \,\, t\ge { t_0},\quad x\in\bar\Omega.
$$
This together with Lemma \ref{persistence-lm1} (2) implies that
\begin{align*}
u_t&=d_1\Delta u-\chi_1 \nabla \cdot (u\nabla w)+u\Big(a_0(t,x)-a_1(t,x)u-a_2(t,x)v\Big)\\
&=d_1\Delta u-\chi_1 \nabla u\cdot\nabla w+u\Big(a_0(t,x)-(a_1(t,x)-\frac{\chi_1 k}{d_3})u-(a_2(t,x)-\frac{\chi_1 l}{d_3})v-\frac{\chi_1 \lambda}{d_3}w\Big)\\
&\ge d_1\Delta u-\chi_1 \nabla u\cdot\nabla w\\
&\,\, +u\Big(a_0(t,x)-(a_1(t,x)-\frac{\chi_1 k}{d_3})u-\big(a_2(t,x)-\frac{\chi_1 l}{d_3}\big)_+(\bar B_2+\epsilon_0)-\frac{\chi_1 \lambda}{d_3}\big(\frac{k}{\lambda} \epsilon_0+\frac{l}{\lambda} (\bar B_2+\epsilon_0)\big)\Big) \\
&\ge d_1\Delta u-\chi_1 \nabla u\cdot\nabla w\\
&\,\, +u\Big(a_{0,\inf}-\big[\big(a_{2,\sup}-\frac{\chi_1 l}{d_3}\big)_++\frac{\chi_1 l}{d_3}\big](\bar B_2+\epsilon_0)-\frac{\chi_1 k}{d_3}\epsilon_0-(a_{1,\sup}-\frac{\chi_1 k}{d_3})u\Big)
\end{align*}
for $0<t\le T$.
Let $\tilde u(t)$ be the solution of
$$
\tilde u_t=\tilde u\Big(a_{0,\inf}-\big[\big(a_{2,\sup}-\frac{\chi_1 l}{d_3}\big)_++\frac{\chi_1 l}{d_3}\big](\bar B_2+\epsilon_0)-\frac{\chi_1 k}{d_3}\epsilon_0-(a_{1,\sup}-\frac{\chi_1 k}{d_3})\tilde u\Big)
$$
with $\tilde u(t_0)=\inf_{ \bar\Omega}u_0(x)$.  We have
$\tilde u(t)$ is monotonically increasing in $t\ge t_0$ and
$$\lim_{t\to\infty} \tilde u(t)=\frac{a_{0,\inf}-\big[\big(a_{2,\sup}-\frac{\chi_1 l}{d_3}\big)_++\frac{\chi_1 l}{d_3}\big](\bar B_2+\epsilon_0)-\frac{\chi_1 k}{d_3}\epsilon_0}{(a_{1,\sup}-\frac{\chi_1 k}{d_3})}.
$$
By comparison principle for parabolic equations, we have
$$
u(x,t+t_0;t_0)\ge \tilde u(t+t_0)> \inf_{ \bar\Omega} u_0(x)\quad \forall\,\, 0<t\le T.
$$

(2)(ii)It can be proved by the similar arguments as those in (2)(i).
\end{proof}

\begin{lemma}
\label{persistence-lm3}
\begin{itemize}
\item[(1)]
Assume (H1). { Let $\epsilon_0$ and $\delta_0=\delta_0(\epsilon_0)$ be such that Lemma
\ref{persistence-lm1}(1) holds with $\epsilon=\epsilon_0$ and $\delta=\delta_0$.}
There are $\underbar A_1^1>0$ and $\underbar A_2^1>0$ such that for any $t_0\in\RR$ and   { $u_0, v_0 \in C^+(\bar{\Omega})$ with}  $0< u_0\le { \bar A_1+\epsilon}$ and $0< v_0\le { \bar A_2+\epsilon}$, { the solution $(u(x,t;t_0)$, $v(x,t;t_0)$, $w(x,t;t_0))$  of $\eqref{u-v-w-eq00} +\eqref{ic}$ satisfies the following.}
\begin{itemize}
\item[(i)] For any $t\ge T$, if $\sup_{ \bar \Omega} u(x,t+t_0;t_0)\ge \delta_0$,  $\inf_{ \bar\Omega} u(x,t+t_0;t_0)\ge \underbar A_1^1$.

\item[(ii)] For any $t\ge T$, if $\sup_{ \bar \Omega} v(x,t+t_0;t_0)\ge \delta_0$,  $\inf_{l \bar\Omega} v(x,t+t_0;t_0)\ge \underbar A_2^1$.
\end{itemize}

\item[(2)] Assume (H2). { Let $\epsilon_0$ and $\delta_0=\delta_0(\epsilon_0)$ be such that Lemma
\ref{persistence-lm1}(2) holds with $\epsilon=\epsilon_0$ and $\delta=\delta_0$.}
There are $\underbar B_1^1>0$ and $\underbar B_2^1>0$ such that for any $t_0\in\RR$ and { $u_0, v_0 \in C^+(\bar{\Omega})$ with} $0< u_0\le \bar B_1+\epsilon$ and $0< v_0\le \bar B_2+\epsilon$, { the solution $(u(x,t;t_0)$, $v(x,t;t_0)$, $w(x,t;t_0))$  of $\eqref{u-v-w-eq00} +\eqref{ic}$ satisfies the following.}
\begin{itemize}
\item[(i)] For any $t\ge T$, if $\sup_{ \bar \Omega} u(x,t+t_0;t_0)\ge \delta_0$, $\inf_{ \bar\Omega} u(x,t+t_0;t_0)\ge \underbar B_1^1$.

\item[(ii)] For any $t\ge T$, if $\sup_{ \bar \Omega} v(x,t+t_0;t_0)\ge \delta_0$,  $\inf_{ \bar\Omega} v(x,t+t_0;t_0)\ge \underbar B_2^1$.
\end{itemize}
\end{itemize}
\end{lemma}

\begin{proof}
(1)(i)  Assume that (1)(i)  does not hold. Then there are $t_{0n}\in\RR$, $t_n\ge T$,  and $u_n,v_n$ with { $0<u_n \le \bar A_1+\epsilon$ and $0<v_n \le  \bar A_2+\epsilon$ } such that
$$
{\sup_{x \in \bar \Omega}} u(x,t_n+t_{0n};t_{0n},u_n,v_n)\ge \delta_0,\quad \lim_{n\to\infty} \inf_{x\in\bar\Omega} u(x,t_n+t_{0n};t_{0n},u_n,v_n)=0.
$$
By Theorem \ref{thm-global-000}(1),
$$
{ 0<u(x,t+t_{0n};t_{0n},u_n,v_n)\le \bar A_1+\epsilon_0,\quad 0<v(x,t+t_{0n};t_{0n},u_n,v_n)\leq \bar A_2+\epsilon_0}\quad \forall\,\, t>{ t_0},\,\,\, x\in\bar\Omega.
$$
Without loss of generality, we may assume that
{
$$
\lim_{n\to\infty} a_i(x,t+t_n+t_{0n})=\tilde a_i(x,t),\quad \lim_{n\to\infty} b_i(x,t+t_n+t_{0n})=\tilde b_i(x,t)
$$}
and
$$
\lim_{n\to\infty} u(x,t+t_n+t_{0n};t_{0n})=\tilde u(x,t),\quad \lim_{n\to\infty} v(x,t+t_n+t_{0n};t_{0n})=\tilde v(x,t)
$$
{ uniformly in $x\in\bar\Omega$ and $t$ in bounded closed sets of $(-T,\infty)$.}
{Note that
$$
u(x,t+t_n+t_{0n};t_{0n},u_n,v_n)=u(x,t+t_n+t_{0n};t_n+t_{0n},u(\cdot,t_n+t_{0n};t_{0n},u_n,v_n),v(\cdot,t_n+t_{0n};t_{0n},u_n,v_n)),
$$
and
$$
v(x,t+t_n+t_{0n};t_{0n},u_n,v_n)=v(x,t+t_n+t_{0n};t_n+t_{0n},u(\cdot,t_n+t_{0n};t_{0n},u_n,v_n),v(\cdot,t_n+t_{0n};t_{0n},u_n,v_n)).
$$
Therefore
$$
\tilde u(x,t)= \tilde u(x,t;0,\tilde u(\cdot,0),\tilde v(\cdot,0)), \quad \tilde v(x,t)= \tilde v(x,t;0,\tilde u(\cdot,0),\tilde v(\cdot,0)),
$$
where $\left( \tilde u(x,t;0,\tilde u(\cdot,0),\tilde v(\cdot,0)),\tilde v(x,t;0,\tilde u(\cdot,0),\tilde v(\cdot,0)), \tilde w(x,t;0,\tilde u(\cdot,0),\tilde v(\cdot,0))\right)$ is the solution of \eqref{u-v-w-eq00} on
$(-T,\infty)$ with $a_i $ being replaced by $\tilde a_i$ and $b_i $ being replaced by $\tilde b_i $, and
$$\big( \tilde u(x,0;0,\tilde u(\cdot,0),\tilde v(\cdot,0)),\tilde v(x,0;0,\tilde u(\cdot,0),\tilde v(\cdot,0))\big)=\big(\tilde u(x,0),\tilde v(x,0)\big).
$$   Moreover,}
{  we have
$$
\tilde u(x,-\frac{T}{2})=\lim_{n\to\infty} u(x,-\frac{T}{2}+t_n+t_{0n};t_{0n},u_n,v_n),
$$
and
$$
\tilde v(x,-\frac{T}{2})=\lim_{n\to\infty} v(x,-\frac{T}{2}+t_n+t_{0n};t_{0n},u_n,v_n).
$$}
Hence $\tilde u(x,-T/2)\ge 0$, $\tilde v(x,-T/2)\ge 0$ for $x\in\bar\Omega$,  and
$\sup_{ \bar\Omega}\tilde u(x,0)\ge \delta_0,\quad \inf_{ \bar\Omega} \tilde u(x,0)=0,$
which is a contradiction  { by  comparison principle for parabolic equations}. Hence (1)(i) holds.

{ (1)(ii), (2)(i), (2)(ii)  can be proved by the similar arguments as those in (1)(i).}
\end{proof}

\begin{lemma}
\label{persistence-lm4}
\begin{itemize}
\item[(1)] { Assume  (H4).  Let $\epsilon_0$ and $\delta_0=\delta_0(\epsilon_0)$ be such that Lemma
\ref{persistence-lm1}(1) and Lemma \ref{persistence-lm2}(1) hold with $\epsilon=\epsilon_0$ and $\delta=\delta_0$.}
There are  $\underbar A_1^2>0$ and $\underbar A_2^2>0$ such that  for any $t_0\in\RR$ and  { $u_0, v_0 \in C^+(\bar{\Omega})$ with} $0< u_0\le { \bar  A_1+\epsilon}$ and $0< v_0\le { \bar A_2+\epsilon}$, { the solution $(u(x,t;t_0)$, $v(x,t;t_0)$, $w(x,t;t_0))$  of $\eqref{u-v-w-eq00} +\eqref{ic}$ satisfies the following.}
\begin{itemize}
\item[(i)] For any $\underbar A_1\le \underbar A_1^2$, if  $\inf_{ \bar\Omega} u_0(x)\ge \underbar A_1$, then $\inf_{ \bar\Omega} u(x,T+t_0;t_0)\ge \underbar A_1$.

\item[(ii)] For any $\underbar A_2\le \underbar A_2^2$, if $\inf_{ \bar\Omega}v_0(x)\ge \underbar A_2$, then $\inf_{ \bar\Omega} v(x,T+t_0;t_0)\ge \underbar A_2$.
\end{itemize}

\item[(2)]
{ Assume  (H5).  Let $\epsilon_0$ and $\delta_0=\delta_0(\epsilon_0)$ be such that Lemma
\ref{persistence-lm1}(2) and  Lemma
\ref{persistence-lm2}(2) hold with $\epsilon=\epsilon_0$ and $\delta=\delta_0$.}
There are $\underbar B_1^1>0$ and $\underbar B_2^1>0$ such that for any $t_0\in\RR$ and { $u_0, v_0 \in C^+(\bar{\Omega})$ with} $0< u_0\le \bar B_1+\epsilon$ and $0< v_0\le \bar B_2+\epsilon$, { the solution $(u(x,t;t_0)$, $v(x,t;t_0)$, $w(x,t;t_0))$  of $\eqref{u-v-w-eq00} +\eqref{ic}$ satisfies the following.}
\begin{itemize}
\item[(i)] For any $\underbar B_1\le \underbar B_1^2$, if  $\inf_{ \bar\Omega}u_0(x)\ge \underbar B_1$, then $\inf_{ \bar\Omega} u(x,T+t_0;t_0)\ge \underbar B_1$.

\item[(ii)] For any $\underbar B_2\le \underbar B_2^2$, if $\inf_{ \bar\Omega}v_0(x)\ge \underbar B_2$, then $\inf_{ \bar\Omega} v(x,T+t_0;t_0)\ge \underbar B_2$.
\end{itemize}

\end{itemize}
\end{lemma}

\begin{proof}
(1)(i)  We prove it using properly modified similar arguments of \cite[Lemma 5.3]{ITBWS16}.

Assume that (1)(i) does not hold. Then there are $\underbar A_{1,n}\to 0$, { $u_n, v_n \in C^+(\bar{\Omega})$  with $0<u_n \le \bar A_1+\epsilon$ and $0<v_n \le  \bar A_2+\epsilon$ },
$t_n\in\RR$, and $x_n\in\Omega$ such that
 $$
u_n(x)\ge \underbar A_{1,n}\quad \forall\,\, x\in\bar\Omega
\quad {\rm and}\quad
u(x_n,T+t_n;t_n,u_n,v_n)<\underbar A_{1,n}.
$$
Let
$$
\Omega_n=\{x\in\Omega\,|\, u_n(x)\ge \frac{\delta_0}{2}\}.
$$
Without loss of generality, we may assume that $\lim_{n\to\infty}|\Omega_n|$ exists. Let
$$
m_0=\lim_{n\to\infty}|\Omega_n|.
$$

{Assume that $m_0=0$. Then there is $\tilde u_n\in C^0(\bar\Omega)$ such that
$$
\underbar A_{1,n}\le \tilde u_n(x)\le  \frac{\delta_0}{2}\quad {\rm and}\quad
\lim_{n\to\infty} \|u_n-\tilde u_n\|_{L^p(\Omega)}=0\quad \forall\,\, 1\le p<\infty.
$$
This implies that
$$
\lim_{n\to\infty} \|\phi^1_n(\cdot,t)\|_{L^p(\Omega)}+\lim_{n\to\infty} \|\phi^2_n(\cdot,t)\|_{L^p(\Omega)}=0
$$
uniformly in $t\in[t_n,t_n+T]$ for all $1\le p<\infty,$ where $\phi^1_n(\cdot,t)=u(\cdot,t;t_n,u_n,v_n)-u(\cdot,t;t_n,\tilde u_n,v_n)$ and $\phi^2_n(\cdot,t)=v(\cdot,t;t_n,u_n,v_n)-v(\cdot,t;t_n,\tilde u_n,v_n).$  Indeed, let
$$G^1_n(\cdot,t)=u(\cdot,t;t_n,u_n,v_n),\,\, G^2_n(\cdot,t)=v(\cdot,t;t_n,u_n,v_n),\,\, W_n(\cdot,t)=w(\cdot,t;t_n, u_n,v_n),
$$
$$ \tilde G^1_n(\cdot,t)=u(\cdot,t;t_n,\tilde u_n,v_n),\,\, \tilde G^2_n(\cdot,t)=v(\cdot,t;t_n,\tilde u_n,v_n),\,\,  \tilde W_n(\cdot,t)=w(\cdot,t;t_n, \tilde u_n,v_n),
 $$
 and
 $$\hat W_n(\cdot,t)(\cdot,t)=w(\cdot,t;t_n, u_n,v_n)-w(\cdot,t;t_n,\tilde u_n,v_n).$$
   Then
\begin{align}
\label{prooflemm3.4-eq1}
\phi^1_n(\cdot,t)=&  e^{-A(t-t_n)}\big(u_n-\tilde u_n\big)-\chi_1\int_{t_n}^t e^{-A(t-s)}\nabla \cdot \big[\phi^1_n(\cdot,s) \nabla W_n(\cdot,s)+\tilde G^1_n(\cdot,s) \nabla \hat W_n(\cdot,s)  \big]ds\nonumber\\
&+\int_{t_n}^t e^{-A(t-s)}\phi^1_n(\cdot,s)\Big(1+ a_0(s,\cdot)-a_1(s,\cdot)(G^1_n(\cdot,s)+\tilde G^1_n(\cdot,s))-a_2(s,\cdot)G^2_n(\cdot,s) \Big)ds\nonumber\\
&-\int_{t_n}^t e^{-A(t-s)}a_2(s,\cdot)\big( \tilde G^1_n(\cdot,t)\big)\phi^2_n(\cdot,s)ds,
\end{align}
and
\begin{align}
\label{prooflemm3.4-eq2}
\phi^2_n(\cdot,t)=&  -\chi_2\int_{t_n}^t e^{-A(t-s)}\nabla\cdot \big[\phi^2_n(\cdot,s) \nabla W_n(\cdot,s)+\tilde G^2_n(\cdot,s) \nabla \hat W_n(\cdot,s)  \big]ds\nonumber\\
&+\int_{t_n}^t e^{-A(t-s)}\phi^2_n(\cdot,s)\Big(1+ b_0(s,\cdot)-b_2(s,\cdot)(G^2_n(\cdot,s)+\tilde G^2_n(\cdot,s))-b_1(s,\cdot)G^1_n(\cdot,s) \Big)ds\nonumber\\
&-\int_{t_n}^t e^{-A(t-s)}b_1(s,\cdot)\big( \tilde G^2_n(\cdot,t)\big)\phi^1_n(\cdot,s)ds,
\end{align}
where $A=-\Delta +I$ with
  $D(A)=\Big\{ u \in W^{2,p}(\Omega) \, |  \, \frac{\p u}{\p n}=0 \quad \text{on } \, \p \Omega \Big\}
  $
(it is known that $A$ is a sectorial operator in { $X=L^p(\Omega)$}).
Now, fix  $1<p <\infty.$  By  regularity and a  priori estimates for
elliptic equations,  \cite[Theorem 1.4.3]{DH77},  \cite[Lemma 2.2]{ITBWS16}, \eqref{prooflemm3.4-eq1}, and \eqref{prooflemm3.4-eq2},   for any $\epsilon \in (0, \frac{1}{2}),$  there is $C=C(\epsilon)>0$ such that
\vspace{-0.1in}\begin{align}
\label{prooflemm3.4-eq3}
&\|\phi^1_n(\cdot,t)\|_{L^p(\Omega)}\nonumber\\
&\leq \|u_n-\tilde u_n\|_{L^p(\Omega)}+ C\chi_1\max_{t_n\leq s\leq t_n+T}\|\nabla W_n(\cdot,s))\|_{ \infty}\int_{t_n}^t(t-s)^{-\epsilon-\frac{1}{2}}\|\phi^1_n(\cdot,s)\|_{L^p(\Omega)}ds\nonumber\\
&\,\, +C\chi\max_{t_n\leq s\leq t_n+T}\|\hat W_n(\cdot,s)\|_{ \infty}\int_{t_n}^t(t-s)^{-\epsilon-\frac{1}{2}}(\|\phi^1_n(\cdot,s)\|_{L^p(\Omega)}+\|\phi^2_n(\cdot,s)\|_{L^p(\Omega)})ds\nonumber\\
&\,\, +C\int_{t_n}^t \{1+ a_{0,sup}+a_{1,\sup}[\max_{t_n\leq s\leq t_n+T}(\| G^1(\cdot,s)\|_{ \infty}+\|\tilde G^1(\cdot,s)\|_{\infty})]\}\|\phi^1_n(\cdot,s)\|_{L^p(\Omega)}ds\nonumber \\
&\,\, +Ca_{2,\sup}\max_{t_n\leq s\leq t_n+T}\| G^2(\cdot,s)\|_{ \infty}\int_{t_n}^t \|\phi^1_n(\cdot,s)\|_{L^p(\Omega)}ds\nonumber\\
&\,\,  +Ca_{2,\sup}\max_{t_n\leq s\leq t_n+T}\|\tilde G^1(\cdot,s)\|_{ \infty}\int_{t_n}^t \|\phi^2_n(\cdot,s)\|_{L^p(\Omega)}ds.
\end{align}
and
\vspace{-0.1in}\begin{align}
\label{prooflemm3.4-eq4}
&\|\phi^2_n(\cdot,t)\|_{L^p(\Omega)}\nonumber\\
&\le C\chi_2\max_{t_n\leq s\leq t_n+T}\|\nabla W_n(\cdot,s))\|_{\infty}\int_{t_n}^t(t-s)^{-\epsilon-\frac{1}{2}}\|\phi^2_n(\cdot,s)\|_{L^p(\Omega)}ds\nonumber\\
&\,\, +C\chi\max_{t_n\leq s\leq t_n+T}\|\hat W_n(\cdot,s)\|_{ \infty}\int_{t_n}^t(t-s)^{-\epsilon-\frac{1}{2}}(\|\phi^1_n(\cdot,s)\|_{L^p(\Omega)}+\|\phi^2_n(\cdot,s)\|_{L^p(\Omega)})ds\nonumber\\
&\,\, +C\int_{t_n}^t \{1+ b_{0,sup}+b_{2,\sup}[\max_{t_n\leq s\leq t_n+T}(\| G^2(\cdot,s)\|_{\infty}+\|\tilde G^2(\cdot,s)\|_{\infty})]\}\|\phi^2_n(\cdot,s)\|_{L^p(\Omega)}ds\nonumber \\
&\,\, +C b_{1,\sup}\max_{t_n\leq s\leq t_n+T}\| G^1(\cdot,s)\|_{ \infty}\int_{t_n}^t\|\phi^2_n(\cdot,s)\|_{L^p(\Omega)}ds\nonumber\\
&\,\,  +Cb_{1,\sup}\max_{t_n\leq s\leq t_n+T}\|\tilde G^2(\cdot,s)\|_{ \infty}\int_{t_n}^t\|\phi^1_n(\cdot,s)\|_{L^p(\Omega)}ds.
\end{align}
Therefore there exists a positive constant $C_0$ independent of  $n$ such that
\begin{align}
\label{prooflemm3.4-eq5}
&\|\phi^1_n(\cdot,t+t_n)\|_{L^p(\Omega)}+\|\phi^1_n(\cdot,t+t_n)\|_{L^p(\Omega)}\nonumber\\
&\leq \|u_n-\tilde u_n\|_{L^p(\Omega)}+ C_0\int_{0}^{t}(t-s)^{-\epsilon-\frac{1}{2}}(\|\phi^1_n(\cdot,s+t_n)\|_{L^p(\Omega)}+\|\phi^1_n(\cdot,s+t_n)\|_{L^p(\Omega)})ds
\end{align}
for all $t\in [0,T]$.
By \eqref{prooflemm3.4-eq5} and the generalized Gronwall's inequality (see \cite[page 6]{DH77}), we get
$$
\lim_{n\to\infty} (\|\phi^1_n(\cdot,t)\|_{L^p(\Omega)}+\|\phi^1_n(\cdot,t)\|_{L^p(\Omega)})=0
$$
uniformly in $t\in[t_n,t_n+T]$ for all $1\leq p<\infty.$
This implies that
$$
\lim_{n\to\infty}\|w(\cdot,t;t_n,u_n,v_n))-w(\cdot,t;t_n;\tilde u_n,v_n)\|_{C^1(\bar \Omega)}=0
$$
uniformly in $t\in[t_n,t_n+T]$. Note that $v(x,t;t_n,\tilde u_n,v_n)\le \bar A_2+\epsilon_0$ for $t\in[t_n,t_n+T]$ and
by Lemma \ref{persistence-lm1}(1), $u(x,t;t_n,\tilde u_n,v_n)\le \epsilon_0$ for $t\in [t_n,t_n+T]$. Hence
$$
w(\cdot,t;t_n;\tilde u_n,v_n)\le  \frac{k}{\lambda}\epsilon_0+\frac{l}{\lambda}(\bar{A_2}+\epsilon_0)
$$
for all $t\in [t_n,t_n+T]$ and $x\in\Omega$. It then follows that for any $\epsilon>0$,
$$
w(\cdot,t;t_n;u_n,v_n)\le  (\frac{k}{\lambda}+\epsilon)\epsilon_0+\frac{l}{\lambda}(\bar{A_2}+\epsilon_0)
$$
for all $t\in [t_n,t_n+T]$, $x\in\Omega$, and $n\gg 1$. Then by the arguments of Lemma \ref{persistence-lm2}, $\inf u(\cdot,t_n+T;t_n,u_n)\ge A_{1,n}$, which
is a contradiction.
Therefore, $m_0\not =0$.

By $m_0\not =0$ and comparison principle for parabolic equations,  without loss of generality, we may assume that
 $$
 \liminf_{n\to\infty}\|e^{-At}u_n\|_{\infty}>0\quad \forall\,\, t\in [0,T].
 $$
 This  implies that there is $0<T_0<T$ and $\delta_\infty>0$ such that
 $$\sup_{x\in\bar\Omega} u(x,t_n+T_0;t_n,u_n,v_n)\ge \delta_\infty$$
 for all $n\gg  1$.
 By  a priori estimates for parabolic equations, without loss of
generality, we may assume that
$$
u(\cdot,t_n+T_0;t_n,u_n,v_n)\to u_0^*,\quad v(\cdot,t_n+T_0;t_n,u_n,v_n)\to v_0^*
$$
and
$$ u(\cdot,t_n+T;t_n,u_n,v_n)\to u^*,\quad v(\cdot,t+n+T;t_n,u_n,v_n)\to v^*
$$
as $n\to\infty$. Without loss of generality, we may also assume that
$$
a_i(t+t_n,x)\to a_i^*(t,x),\quad b_i(t+t_n,\cdot)\to b_i^*(t,x)
$$
as $n\to\infty$ locally uniformly in $(t,x)\in\RR\times\bar\Omega$. Then  we have
$$
u^*(x)=u^*(x,T;T_0,u_0^*,v_0^*),\quad v^*(x)=v^*(x,T;t_0,u_0^*,v_0^*)
$$
and
$$
 \inf_{\bar\Omega} u^*(x)=0,\quad  \inf_{ \bar\Omega }v^*(x)\ge 0,
$$
where $(u^*(x,t;T_0,u_0^*,v_0^*), v^*(x,t;T_0,u_0^*,v_0^*),w(x,t;T_0,u_0^*,v_0^*))$  is the solution of \eqref{u-v-w-eq00} with
$a_i(t,x)$ and $b_i(t,x)$ being replaced by $a_i^*(t,x)$ and $b_i^*(t,x)$, and $(u^*(x,T_0;T_0,u_0^*,v_0^*), v^*(x,T_0;T_0,u_0^*,v_0^*))
=(u_0^*(x),v_0^*(x))$.
By comparison principle, we must have $u_0^*\equiv 0$. But
$\sup_{\bar\Omega} u_0^*\ge   \delta_\infty.$
This is a contradiction.
}

(1)(ii)  It can be proved by the similar arguments as those in (1)(i).

(2) Follows by similar arguments as those in (1).
\end{proof}

Let
\begin{equation}
\label{I-underbar-eq}
\underbar A_1=\min\{\underbar A_1^1,\underbar A_1^2\},\quad \underbar A_2=\min\{\underbar A_2^1,\underbar A_2^2\}
\end{equation}
and
\begin{equation}
\label{A-underbar-eq}
\underbar B_1=\min\{\underbar B_1^1,\underbar B_1^2\},\quad \underbar B_2=\min\{\underbar B_2^1,\underbar B_2^2\}.
\end{equation}

\medskip
{ Note that the constants $\underbar A_1$,  $\underbar A_2$,  $\underbar B_1$ and $\underbar B_2$ depend on $T$ and $\epsilon_0$.
}

\begin{lemma}
\label{persistence-lm5}
\begin{itemize}
\item[(1)] { Assume (H4).  Let $\epsilon_0$ and $\delta_0=\delta_0(\epsilon_0)$ be such that Lemma
\ref{persistence-lm1}(1)  and  Lemma
\ref{persistence-lm2}(1) hold with $\epsilon=\epsilon_0$ and $\delta=\delta_0$.}
For any  { $u_0, v_0 \in C^+(\bar{\Omega})$ with} $0< u_0\le {\bar A_1+\epsilon}$ and $0< v_0\le {\bar A_2+\epsilon}$, {the solution $(u(x,t;t_0)$, $v(x,t;t_0)$, $w(x,t;t_0))$  of $\eqref{u-v-w-eq00} +\eqref{ic}$ satisfies the following.}
\begin{itemize}
\item[(i)] If $\inf_{ \bar\Omega}u_0(x)\ge \underbar A_1$, then
\begin{equation}
\label{lower-bound-eq1}
\underbar A_1\le u(x,t+t_0;t_0)\le \bar A_1+\epsilon_0\quad \forall\,\, t\ge T,\,\,\, x\in\bar\Omega.
\end{equation}

\item[(ii)]  If $\inf_{ \in\bar\Omega}v_0(x)\ge \underbar A_2$, then
\begin{equation}
\label{lower-bound-eq2}
\underbar A_2\le v(x,t+t_0;t_0)\le \bar A_2+\epsilon_0\quad \forall\,\, t\ge T,\,\,\, x\in\bar\Omega.
\end{equation}
\end{itemize}

\item[(2)]
{Assume  (H5).  Let $\epsilon_0$ and $\delta_0=\delta_0(\epsilon_0)$ be such that Lemma
\ref{persistence-lm1}(2) and Lemma
\ref{persistence-lm2}(2) hold with $\epsilon=\epsilon_0$ and $\delta=\delta_0$.}
For any $t_0\in\RR$ and {$u_0, v_0 \in C^+(\bar{\Omega})$ with} $0< u_0\le \bar B_1+\epsilon$ and $0< v_0\le \bar B_2+\epsilon$, { the solution $(u(x,t;t_0)$, $v(x,t;t_0)$, $w(x,t;t_0))$  of $\eqref{u-v-w-eq00} +\eqref{ic}$ satisfies the following.}
\begin{itemize}
\item[(i)] If $\inf_{\in\bar\Omega}u_0(x)\ge \underbar B_1$, then
\begin{equation}
\label{lower-bound-eq1-2}
\underbar B_1\le u(x,t+t_0;t_0)\le \bar B_1+\epsilon_0\quad \forall\,\, t\ge T,\,\,\, x\in\bar\Omega.
\end{equation}

\item[(ii)]  If $\inf_{\bar\Omega}v_0(x)\ge \underbar B_2$, then
\begin{equation}
\label{lower-bound-eq2-2}
\underbar B_2\le v(x,t+t_0;t_0)\le \bar B_2+\epsilon_0\quad \forall\,\, t\ge T,\,\,\, x\in\bar\Omega.
\end{equation}
\end{itemize}
\end{itemize}
\end{lemma}

\begin{proof}
(1)(i) First of all, by Lemma \ref{persistence-lm4}(1), we have
$$
\underbar A_1\le u(x,T+t_0;t_0)\le \bar A_1+\epsilon_0\quad \forall\,\, x\in\bar\Omega.
$$
Note that we have
$$
{\rm either}\,\,\, \sup_{\bar \Omega} u(x,T+t_0;t_0)> \delta_0\,\,\, {\rm or}\,\,\, \sup_{\bar \Omega} u(x,T+t_0;t_0)
  \le \delta_0.
  $$
   In the former case, if $\sup_{\bar \Omega} u(x,t+T+t_0;t_0)> \delta_0$ for all
  $0\le t\le T$, by Lemma \ref{persistence-lm3}, \eqref{lower-bound-eq1} holds
   for all $T\le t\le 2T$. If there is $t^*\in (T,2T)$ such that
   $\sup_{\bar \Omega} u(x,t+t_0;t_0)> \delta_0$ for $T\le t\le t^*$ and $\sup_{\bar \Omega} u(x,t^*+t_0;t_0)= \delta_0$, then
   by Lemma  \ref{persistence-lm3}, \eqref{lower-bound-eq1} holds
   for all $T\le t\le t^*$,  which together with  Lemma \ref{persistence-lm2}
   implies that  \eqref{lower-bound-eq1} also holds
   for all $t^*\le t\le 2T$.
  In the later case, by Lemma \ref{persistence-lm2},
     \eqref{lower-bound-eq1} also holds
   for all $T\le t\le 2T$.  Therefore, in any case,  \eqref{lower-bound-eq1} also holds
   for all $T\le t\le 2T$. Repeating the above process, we have that \eqref{lower-bound-eq1} also holds
   for all $t\ge T$.

(1)(ii) It can be proved by the similar arguments as those in (1)(i).

(2) It follows from the similar arguments as those in (1).
\end{proof}

We now prove  Theorem \ref{thm-entire-001}.

\begin{proof}[Proof of Theorem \ref{thm-entire-001}]
(1) Let $\epsilon_0$ and $\delta_0=\delta_0(\epsilon_0)$ be such that Lemma
\ref{persistence-lm1}(1) and Lemma \ref{persistence-lm2}(1) hold with $\epsilon=\epsilon_0$ and $\delta=\delta_0$.
Let $\underbar A_1$, $\bar A_1$, $\underbar A_2$, and $\bar A_2$ be as in Lemma \ref{persistence-lm5}(1).
By the assumption that $u_0\not \equiv 0$, $v_0\not\equiv 0$, and
comparison principle for parabolic equations, without loss of generality, we may assume that
$\inf_{\bar\Omega}u_0(x)>0$ and $\inf_{\bar\Omega} v_0(x)>0$.

First, by Theorem \ref{thm-global-000}, there is $T_1=T_1(u_0,v_0,\epsilon_0)$ such that
$$
u(x,t+t_0;t_0)\le \bar A_1+\epsilon_0,\,\, \, v(x,t+t_0;t_0)\le \bar A_2+\epsilon_0
\quad \forall \,\, t\ge T_1,\quad x\in\bar\Omega.
$$
Observe that if $\sup_{\bar\Omega} u(x,t+t_0;t_0)<\delta_0$, then
\begin{align*}
u_t&=d_1\Delta u-\chi_1 \nabla \cdot (u\nabla w)+u\Big(a_0(t,x)-a_1(t,x)u-a_2(t,x)v\Big)\\
&=d_1\Delta u-\chi_1 \nabla u\cdot\nabla w+u\Big(a_0(t,x)-(a_1(t,x)-\frac{\chi_1 k}{d_3})u-(a_2(t,x)-\frac{\chi_1 l}{d_3})v-\frac{\chi_1 \lambda}{d_3}w\Big)\\
&\ge d_1\Delta u-\chi_1 \nabla u\cdot\nabla w\\
&\,\, +u\Big(a_0(t,x)-(a_1(t,x)-\frac{\chi_1 k}{d_3})u-(a_2(t,x)-\frac{\chi_1 l}{d_3})(\bar A_2+\epsilon_0)-\frac{\chi_1 \lambda}{d_3}(\frac{k}{\lambda} \delta_0+\frac{l}{\lambda} (\bar A_2+\epsilon_0))\Big) \\
&\ge d_1\Delta u-\chi_1 \nabla u\cdot\nabla w+u\Big(a_0(t,x)-a_2(t,x)(\bar A_2+\epsilon_0)-\frac{\chi_1 k}{d_3}\epsilon_0-(a_1(t,x)-\frac{\chi_1 k}{d_3})u\Big).
\end{align*}
 Let $\tilde u(t;\tilde u_0)$ be the solution of
$$
\tilde u_t=\tilde u\Big(a_{0,\inf}-a_{2,\sup}(\bar A_2+\epsilon_0)-\frac{\chi_1 k}{d_3}\epsilon_0-(a_{1,\sup}-\frac{\chi_1 k}{d_3})\tilde u\Big)
$$
with $\tilde u(0;\tilde u_0)=\tilde u_0\in (0,\delta_0)$.  We have
$\tilde u(t)$ is monotonically increasing in $t\ge 0$ and
\begin{equation}
\label{new-add-eq1}
\lim_{t\to\infty} \tilde u(t;\tilde u_0)=\frac{a_{0,\inf}-a_{2,\sup}(\bar A_2+\epsilon_0)-\frac{\chi_1 k}{d_3}\epsilon_0}{(a_{1,\sup}-\frac{\chi_1 k}{d_3})}
>\delta_0.
\end{equation}
Observe also that
\begin{equation}
\label{new-add-eq2}
\inf_{t_0\in\RR} \inf_{x \in \bar\Omega}u(x,T+t_0;t_0)>0.
\end{equation}
 Indeed, we have either $\sup_{\bar\Omega} u_0 <\delta_0 $ or $\sup_{\bar\Omega} u_0 \geq \delta_0 $. If $\sup_{\bar\Omega} u_0 <\delta_0 $, we have by Lemma \ref{persistence-lm2} (i) that $ \inf_{\bar\Omega}u(x,T+t_0;t_0)\geq \inf_{\bar\Omega} u_0 >0$ for all $t_0\in \RR$ and then \eqref{new-add-eq2} follows.
If $\sup_{\bar\Omega} u_0\ge \delta_0$, but \eqref{new-add-eq2} does not hold, then there are $t_{0n}\in\RR$ and $x_n\in\bar\Omega$ such that
$$
\lim_{n\to\infty} u(x_n,T+t_{0n};t_{0n},u_0,v_0)=0.
$$
Let $a_i^n(t,x)=a_i(t+t_{0n},x)$ and $b_i^n(t,x)=b_i(t+t_{0n},x)$ for $i=0,1,2$. Then
\begin{align*}
&(u(x,t+t_{0n};t_{0n},u_0,v_0),v(x,t+t_{0n};t_{0n},u_0,v_0),w(x,t+t_{0n};t_{0n},u_0,v_0))\\
&=(u^n(x,t;u_0,v_0),v^n(x,t;u_0,v_0),w^n(x,t;u_0,v_0))
\end{align*}
 for $t\ge 0$, where
$(u^n(x,t;u_0,v_0),v^n(x,t;u_0,v_0),w^n(x,t;u_0,v_0))$ is the solution of \eqref{u-v-w-eq00} with $a_i$ and $b_i$ ($i=0,1,2$) being replaced by
$a_i^n$ and $b_i^n$ ($i=0,1,2$) and  $(u^n(x,0;u_0,v_0),v^n(x,0;u_0,v_0))=(u_0(x),v_0(x))$. Without loss of generality, we may assume that
$$
\lim_{n\to\infty}a_i^n(t,x)=a_i^\infty(t,x),\quad \lim_{n\to\infty} b_i^n(t,x)=b_i^\infty(t,x)
$$
uniformly in $x\in\bar\Omega$ and $t$ in bounded sets of $\RR$, and
$$
\lim_{n\to\infty} x_n=x_\infty.
$$
 Then
\begin{align*}
&\lim_{n\to\infty} (u^n(x,t;u_0,v_0),v^n(x,t;u_0,v_0),w^n(x,t;u_0,v_0))\\
&=(u^\infty(x,t;u_0,v_0),v^\infty(x,t;u_0,v_0),w^\infty(x,t;u_0,v_0))
\end{align*}
uniformly in $x\in\bar\Omega$ and $t$ in bounded set of $[0,\infty)$, where $(u^\infty(x,t;u_0,v_0),v^\infty(x,t;u_0,v_0)$, $w^\infty(x,t;u_0,v_0))$ is the solution of \eqref{u-v-w-eq00} with $a_i$ and $b_i$ ($i=0,1,2$) being replaced by
$a_i^\infty$ and $b_i^\infty$ ($i=0,1,2$) and  $(u^\infty(x,0;u_0,v_0),v^\infty(x,0;u_0,v_0))=(u_0(x),v_0(x))$.
It then follows that
$$
\inf_{\bar\Omega} u_0(x)>0\quad {\rm and}\quad u^\infty(x_\infty,T;u_0,v_0)=0,
$$
which is a contradiction. Hence if   $\sup_{\bar\Omega} u_0\ge \delta_0$,  \eqref{new-add-eq2} also holds.

{ Note that we have either
$\sup _{\bar\Omega} u(x,T+t_0;t_0)\ge \delta_0$ or $\sup _{\bar\Omega} u(x,T+t_0;t_0)< \delta_0$.}
{ If $\sup_{\bar\Omega} u(x,T+t_0;t_0)<\delta_0$, by \eqref{new-add-eq1}, \eqref{new-add-eq2}, and comparison principle for
parabolic equations, there are $ T_2(u_0,v_0, \epsilon_0)\ge T$ and $T\le\tilde  T_2(u_0,v_0,\epsilon_0)\le  T_2(u_0,v_0,\epsilon_0)$ such that
$$
\sup_{\bar\Omega} u(x,\tilde T_2(u_0,v_0,\epsilon_0)+t_0;t_0)=\delta_0.
$$
Hence, in either case, there is $\tilde T_2(u_0,v_0,\epsilon_0)\in [T, T_2(u_0,v_0,\epsilon_0)]$ such that
\begin{equation}
\label{new-add-eq3}
\sup_{\bar\Omega} u(x, \tilde T_2(u_0,v_0,\epsilon_0)+t_0;t_0,) \geq\delta_0.
\end{equation}}
{By \eqref{new-add-eq3} and  Lemma \ref{persistence-lm3},
$$
\inf_{\bar\Omega}u(x,\tilde T_2(u_0,v_0,\epsilon_0)+t_0;t_0)\ge \underbar A_1.
$$}
Then by Lemma \ref{persistence-lm5}(1),
\begin{equation}
\label{thm2-proof-eq1}
\underbar A_1\le u(x,t+t_0;t_0)\le \bar A_1+\epsilon_0\quad \forall \,\, t\ge \max\{T_1(u_0,v_0,\epsilon_0),{ T+T_2(u_0,v_0,\epsilon_0)}\}.
\end{equation}

Similarly, we can prove that there are $\tilde T_1(u_0,v_0,\epsilon_0)>0$ and $\tilde T_2(u_0,v_0,\epsilon_0)\ge T$ such that
\begin{equation}
\label{thm-proof-eq2}
\underbar A_2\le v(x,t+t_0;t_0)\le \bar A_2+\epsilon_0\quad \forall \,\, t\ge \max\{\tilde T_1(u_0,v_0,\epsilon_0),T+ \tilde T_2(u_0,v_0,\epsilon_0)\}.
\end{equation}
By  Theorem \ref{thm-global-000}, \eqref{thm2-proof-eq1}, and \eqref{thm-proof-eq2},
 for any $\epsilon>0$, there is $t_{\epsilon,u_0,v_0}$ such that \eqref{attracting-set-eq00} holds.

(2)  It follows from the similar arguments as those in (1).
\end{proof}

\begin{corollary}
\label{u-w-cor}
Consider \eqref{u-w-eq00} and assume \eqref{u-w-cond}. There is $\underbar A_1$  such that for  any $\epsilon>0,$ $t_0\in\RR,$ $u_0\in C^0(\bar \Omega)$ with $u_0\ge 0,$ and $u_0\not \equiv 0$, there exists $t_{\epsilon,u_0}$ such that
\begin{equation*}
\underbar A_1 \le u(x,t;t_0,u_0) \le \bar{A_1}+\epsilon
\end{equation*}
for all $x\in\bar\Omega$ and $t\ge t_0+t_{\epsilon,u_0}$, where $(u(x,t;t_0,u_0),w(x,t;t_0,u_0))$ is the global solution of
\eqref{u-w-eq00} with $u(x,t_0;t_0,u_0)=u_0(x)$
\end{corollary}

\begin{proof}{  We outline  the proof in the following 6 steps.

\smallskip
{\bf Step 1.} Fix $T>0.$ By the arguments of Lemma \ref{persistence-lm1} (1), we have that, for any $\epsilon>0$,
there is $\delta=\delta(\epsilon,T)>0$ such that for any $u_0\in C^+(\bar\Omega)$, if $0\le u_0\le \delta$,
then $u(x,t;t_0,u_0)\le \epsilon$ for $t\in [t_0,t_0+T]$ and $x\in\bar\Omega$.

\smallskip

{\bf Step 2.}  By the arguments of Lemma \ref{persistence-lm2} (1), the following holds.
Let $\epsilon_0$ and $\delta_0=\delta_0(\epsilon_0,T)$ be such that {\bf Step 1} holds with $\epsilon=\epsilon_0$ and $\delta=\delta_0$, and
$$
a_{0,\inf}>\frac{\chi_1 k}{d_3}\epsilon_0\quad {\rm and}\quad
\delta_0<\frac{a_{0,\inf}-\frac{\chi_1 k}{d_3}\epsilon_0}{a_{1,\sup}-\frac{\chi_1 k}{d_3}}.
$$
For given $u_0 \in C^+(\bar\Omega)$,  if $0<u_0<\delta_0$, then
$u(x,t+t_0;t_0,u_0)>\inf u_0(x)\quad \forall \,\, 0<t\le T.$

\smallskip

{\bf Step 3.} By the arguments of Lemma \ref{persistence-lm3} (1), the following holds.
Let $\epsilon_0$ and $\delta_0=\delta_0(\epsilon_0,T)$ be such that {\bf Step 2} holds with $\epsilon=\epsilon_0$ and $\delta=\delta_0$.
There is $\underbar A_1^1>0$  such that for any $t_0\in\RR$ and  $u_0 \in C^+(\bar\Omega)$,  for any $t\ge T$, if $\sup_{\bar \Omega} u(x,t+t_0;t_0,u_0)\ge \delta_0$, then $\inf_{\bar\Omega} u(x,t+t_0;t_0,u_0)\ge \underbar A_1^1$.

\smallskip

{\bf Step 4.} By the arguments of Lemma \ref{persistence-lm4} (1), the following holds.
 Let $\epsilon_0$ and $\delta_0=\delta_0(\epsilon_0)$ be such that {\bf Steps 1 and 2} hold with $\epsilon=\epsilon_0$ and $\delta=\delta_0$.
There is  $\underbar A_1^2>0$  such that  for any $t_0\in\RR$ and $0<u_0 \in C^+(\bar\Omega)$,
for any $\underbar A_1\le \underbar A_1^2$, if  $\inf_{\bar\Omega} u_0(x)\ge \underbar A_1$, then $\inf_{\bar\Omega} u(x,T+t_0;t_0,u_0)\ge \underbar A_1$.

\smallskip

{\bf Step 5.}   By the arguments of Lemma \ref{persistence-lm5} (1), the following holds.
 Let $\epsilon_0$ and $\delta_0=\delta_0(\epsilon_0,T)$ be such that  {\bf Steps 1 and 2}  hold with $\epsilon=\epsilon_0$ and $\delta=\delta_0$.
For any  $0<u_0 \in C^+(\bar\Omega) $,
if $\inf_{\bar\Omega}u_0(x)\ge \underbar A_1$, then
\begin{equation*}
\label{u-w-lower-bound-eq1}
\underbar A_1\le u(x,t+t_0;t_0,u_0)\le \bar A_1+\epsilon_0\quad \forall\,\, t\ge T,\,\,\, x\in\bar\Omega.
\end{equation*}

\smallskip

{\bf Step 6.} Complete the proof by combining  {\bf Step 5} and  the arguments of equation \eqref{thm2-proof-eq1} in the proof of Theorem \ref{thm-entire-001}.}
\end{proof}

\section{Coexistence}
\label{Coexistenec}

In this section, we study the existence of coexistence states in \eqref{u-v-w-eq00} and prove Theorem \ref{thm-entire-002}.

We first prove a lemma.

\begin{lemma}
\label{persistence-lm6}
Consider
\begin{equation}
\begin{cases}
\label{u-v-ode}
u_t=u\big(a_0(t)-a_1(t)u-a_2(t)v\big)\cr
v_t=v\big(b_0(t)-b_1(t)u-b_2(t)v\big).
\end{cases}
 \end{equation}
Assume \eqref{stability-cond-1-eq1} is satisfied.
 Then there is a  positive entire solution $(u^{**}(t),v^{**}(t))$ of \eqref{u-v-ode}. Moreover, for any $u_0,v_0>0$ and $t_0\in\RR$,
 $$
 (u(t;t_0,u_0,v_0),v(t;t_0,u_0,v_0))-(u^{**}(t),v^{**}(t))\to 0
 $$
 as $t\to\infty$, where $(u(t;t_0,u_0,v_0)$, $v(t;t_0,u_0,v_0))$ is the solution of \eqref{u-v-ode} with
 $(u(t_0;t_0,u_0,v_0)$, $v(t_0;t_0,u_0,v_0))=(u_0,v_0)$.
 In addition, if $a_i(t)$ and $b_i(t)$ are almost periodic, then so is $(u^{**}(t)$, $v^{**}(t))$.
\end{lemma}

\begin{proof}
First, let
$$s_1=\frac{b_{2,\inf}a_{0,\inf}-a_{2,\sup}b_{0,\sup}}{b_{2,\inf}a_{1,\sup}-a_{2,\sup}b_{1,\inf}},\quad r_1=\frac{b_{2,\sup}a_{0,\sup}-a_{2,\inf}b_{0,\inf}}{b_{2,\sup}a_{1,\inf}-a_{2,\inf}b_{1,\sup}},$$
 and
 $$ r_2=\frac{a_{1,\inf}b_{0,\inf}-b_{1,\sup}a_{0,\sup}}{a_{1,\inf}b_{2,\sup}-b_{1,\sup}a_{2,\inf}},\quad s_2=\frac{a_{1,\sup}b_{0,\sup}-b_{1,\inf}a_{0,\inf}}{a_{1,\sup}b_{2,\inf}-b_{1,\inf}a_{2,\sup}}.
 $$
  Then
  $$0<s_1\leq r_1 \quad \text{and} \quad 0<r_2\leq s_2.$$

   Next, for given $t_0 \in \mathbb{R}$ and $u_0,v_0\in\RR$,  if $0<u_0\leq r_1$ and ${ v_0\geq  r_2}$,  by
   \cite[Lemma 3.1]{Ahm}, we have
   \begin{equation}
   \label{aux-existence-eq1}
   0<u(t;t_0,u_0,v_0)\leq r_1\quad \text{and}\quad v(t;t_0,u_0,v_0)\geq { r_2}\quad \forall\,\, t\ge t_0.
   \end{equation}
   And if $u_0\geq s_1$ and $0<v_0\leq { s_2}$, by \cite[Lemma 3.2]{Ahm} again,
   \begin{equation}
   \label{aux-existence-eq2}
   u(t;t_0,u_0,v_0)\geq s_1\quad  \text{and} \quad 0<v(t;t_0,u_0,v_0)\leq { s_2}\quad \forall\,\, t\ge t_0.
   \end{equation}
{ Next, by the pullback method, there exists a positive entire solution of \eqref{u-v-ode} which satisfies  $s_1\leq u(t)\leq r_1 \,\, \text{and}\,\,s_2\leq v(t)\leq r_2,$ for all $t \in \mathbb{R}.$ We omit the proof here because a similar proof will be given in Theorem \ref{thm-entire-002}}

Finally, we prove the stability of positive entire solutions and the almost periodicity of positive entire solutions when the coefficients are almost periodic.
Let $(u^{**}(t),v^{**}(t))$ be a positive entire solution of \eqref{u-v-ode} and  let $u_0,v_0>0$ and $t_0\in\RR.$  It follows from \cite[Theorem 1]{Ahm} that
 $$
 (u(t;t_0,u_0,v_0),v(t;t_0,u_0,v_0))-(u^{**}(t),v^{**}(t))\to 0 \quad \text{as} \,\, \to \infty.
 $$
 By \cite[Theorem C]{HeSh}, when $a_i(t)$ and $b_i(t)$ ($i=0,1,2$) are almost periodic in $t$, then positive entire solutions of \eqref{u-v-ode} are unique
 and almost periodic.
 The lemma thus follows.
\end{proof}

We now prove Theorem \ref{thm-entire-002}. Let $T>0$ be fixed and $\underbar A_i$, $\bar A_i$, $\underbar B_i$, and $\bar B_i$ ($i=1,2$) be as in the previous section.

\begin{proof}[Proof of Theorem \ref{thm-entire-002}]
(1) We  first prove the  existence of positive entire solutions. { Let  $u_0, v_0 \in C^0(\bar \Omega)$ be such that $0<\underbar A_1\leq u_0(x)\leq \bar A_1 \, \text{and} \,0<\underbar A_2\leq v_0(x)\leq \bar A_2.$   By Theorem \ref{thm-global-000}(1) and Lemma \ref{persistence-lm5}(1),
\begin{equation}
\label{proof-entire-eq0}
0<\underbar A_1\leq u(x,t+t_0;t_0,u_0,v_0)\leq \bar A_1 \quad \text{and} \quad 0<\underbar A_2\leq v(x,t+t_0;t_0,u_0,v_0)\leq \bar A_2
 \end{equation}
 for all  $x \in \bar \Omega$,  $t\geq T$, and $t_0 \in \mathbb{R}$.
For $n \in \mathbb N$ with $n>T$, set $t_n=-n,$  $u_n=u(\cdot,0;t_n,u_0,v_0)$ and $v_n=v(\cdot,0;t_n,u_0,v_0).$  Then by parabolic regularity there exist $t_{n_k} \in \mathbb{N},$ $u^{**}_0, \, v^{**}_0 \in C^0(\bar{\Omega})$ such that $$u_{n_k} \to u^{**}_0 \quad \text{and} \quad  v_{n_k} \to v^{**}_0 \quad \text{in}\,\, C^0(\bar{\Omega}).$$ We have $u(\cdot,t;t_{n_k},u_0,v_0)=u(\cdot,t;0,u(\cdot,0;t_{n_k},u_0,v_0),v(\cdot,0;t_{n_k},u_0,v_0)),$ and $v(\cdot,t;t_{n_k},u_0,v_0)=v(\cdot,t;0,u(\cdot,0;t_{n_k},u_0,v_0),v(\cdot,0;t_{n_k},u_0,v_0)).$ Thus for $t\geq 0$ we have
 $$ (u(\cdot,t;t_{n_k},u_0,v_0),v(\cdot,t;t_{n_k},u_0,v_0)) \to (u(\cdot,t;0,u^{**}_0,v^{**}_0),v(\cdot,t;0,u^{**}_0,v^{**}_0)) \, \text{in} \, C^0(\bar{\Omega})\times C^0(\bar{\Omega}) .$$ Moreover
 $$
0<\underbar A_1\leq u(x,t;0,u^{**}_0,v^{**}_0)\leq \bar A_1 \quad \text{and} \quad 0<\underbar A_2\leq v(x,t;0,u^{**}_0,v^{**}_0)\leq \bar A_2\quad \forall\,\, x \in \Omega,\,\, t \geq 0.
$$

We now prove that $(u(\cdot,t;0,u^{**}_0,v^{**}_0),v(\cdot,t;0,u^{**}_0,v^{**}_0))$ has backward extension.  In  order to prove that, fix $m \in \mathbb{N}$ and define $u^m_n=u(\cdot,-m;t_n,u_0,v_0)$ and $v^m_n=v(\cdot,-m;t_n,u_0,v_0)$ for all $n>m +T.$ Then by parabolic regularity, without loss of generality,
we may assume that  there exist  $u^{**}_m, \, v^{**}_m \in C^0(\bar{\Omega})$ such that
$$
u^m_{n_k} \to u^{**}_m \quad \text{and} \quad  v^m_{n_k} \to v^{**}_m \quad \text{in}\,\, C^0(\bar{\Omega}).
$$
 Furthermore we have $u(\cdot,t;t_{n_k},u_0,v_0)=u(\cdot,t;-m,u(\cdot,-m;t_{n_k},u_0,v_0),v(\cdot,-m;t_{n_k},u_0,v_0)),$ and $v(\cdot,t;t_{n_k},u_0,v_0)=u(\cdot,t;-m,u(\cdot,-m;t_{n_k},u_0,v_0),v(\cdot,-m;t_{n_k},u_0,v_0)).$ Therefore  we have
 $$
 (u(\cdot,t;t_{n_k},u_0,v_0),v(\cdot,t;t_{n_k},u_0,v_0)) \to (u(\cdot,t;-m,u^{**}_m,v^{**}_m),v(\cdot,t;-m,u^{**}_m,v^{**}_m)) \, \text{in} \, C^0(\bar{\Omega})\times C^0(\bar{\Omega})
  $$
  for all $t\geq -m$ , which implies that $(u(\cdot,t;0,u^{**}_0,v^{**}_0),v(\cdot,t;0,u^{**}_0,v^{**}_0))$ has backward extension  in the sense that
$$
(u(\cdot,t;0,u^{**}_0,v^{**}_0),v(\cdot,t;0,u^{**}_0,v^{**}_0))=(u(\cdot,t;-m,u^{**}_m,v^{**}_m),v(\cdot,t;-m,u^{**}_m,v^{**}_m))
$$
 for all { $t> -m$ } and $m \in \mathbb{N}.$ Moreover
$$
0<\underbar A_1\leq u(\cdot,t;-m,u^{**}_m,v^{**}_m)\leq \bar A_1 \quad \text{and} \quad 0<\underbar A_2\leq v(\cdot,t;-m,u^{**}_m,v^{**}_m)\leq \bar A_2\,\,\, \forall\, x \in \Omega,\,\, t \geq -m.
$$
Set { $u^{**}(x,t)=u(x,t;0,u^{**}_0,v^{**}_0)$, $v^{**}(x,t)=v(x,t;0,u^{**}_0,v^{**}_0)$}, and $w^{**}=(-\Delta+I)^{-1}(ku^{**}+lv^{**}).$
 Then $(u^{**}(x,t),v^{**}(x,t),w^{**}(x,t))$ is a positive bounded entire solution of \eqref{u-v-w-eq00}.
}

\medskip

(i) Assume that $a_i(t+T,x)=a_i(t,x)$ and  $b_i(t+T,x)=b_i(t,x)$ for $i=0,1,2$. Set
  \begin{equation}
  \label{entire-solut-eq1}
  E(T)=\{(u_0,v_0) \in C^0(\bar{\Omega})\times C^0(\bar{\Omega}) \,|\, 0<\underbar A_1\leq u_0(x)\leq \bar A_1 \, \text{and} \,0<\underbar A_2\leq v_0(x)\leq \bar A_2\}.
    \end{equation}
    Note that $E$ is nonempty, closed, convex and bounded subset of $C^0(\bar{\Omega})\times C^0(\bar{\Omega}).$ Define  the map
    $  \mathcal{T}(T):E(T) \to C^0(\bar{\Omega})\times C^0(\bar{\Omega})$ by
    $$\mathcal{T}(T)(u_0,v_0)=(u(\cdot,T;0,u_0,v_0),v(\cdot,T;0,u_0,v_0)).
    $$
Note that $\mathcal{T}(T)$ is well defined, $\mathcal{T}(T) E(T) \subset E(T),$ and continuous by continuity with respect to initial conditions. Moreover by regularity and Arzella-Ascoli's Theorem, $\mathcal{T}(T)$ is completely continuous and therefore by Schauder fixed point there exists $(u_T,v_T) \in E(T) $ such that $(u(\cdot,T;0,u_T,v_T),v(\cdot,T;0,u_T,v_T))=(u_T,v_T).$ Then $((u(\cdot,t;0,u_T,v_T),v(\cdot,t;0,u_T,v_T)$, $w(\cdot,t;0,u_T,v_T)))$ is a positive periodic  solution of \eqref{u-v-w-eq00} with periodic $T.$

(ii)  Assume that $a_i(t,x)\equiv a_i(x)$ and $b_i(t,x)\equiv a_i(x)$ $(i=0,1,2$). In this case, each $\tau>0$ is a period for $a_i$ and $b_i$ . By (i), there exist $(u^\tau,v^\tau) \in E(\tau)$ such that $(u(\cdot,t;0,u^\tau,v^\tau),v(\cdot,t;0,u^\tau,v^\tau)$, $w(\cdot,t;0,u^\tau,v^\tau))$
is a positive  periodic solution of \eqref{u-v-w-eq00} with period $\tau$.

Observe that $C^0(\bar\Omega)\subset L^p(\Omega)$ for any $1\le p<\infty$. Choose
$p>1$ and $\alpha\in (1/2,1)$ are such that $X^\alpha \hookrightarrow C^1(\bar\Omega)$, where $X^\alpha =D(A^\alpha)$ with the graph norm
$\|u\|_\alpha=\|A^\alpha u\|_{L^p(\Omega)}$ and $A=I-\Delta$ with domain $D(A)=\{u\in W^{2,p}(\Omega)\,|\, \frac{\p u}{\p n}=0$ on $\p\Omega\}$.

Note that there is $\tilde M>0$ such that  for each $\tau>0$ and  $(u_0,v_0) \in E(\tau),$ $\|u(\cdot,t;0,u_0,v_0)\|_{\alpha}+\|v(\cdot,t;0,u_0,v_0)\|_{\alpha} \leq \tilde M $ for each $1\leq t \leq 2.$
  Let $\tau_n=\frac{1}{n},$ then  there exists $u_n,v_n \in E(\frac{1}{n})$ such that $(u(\cdot,t;0,u_n,v_n),v(\cdot,t;0,u_n,v_n)$, $w(\cdot,t;0,u_n,v_n))$ is periodic with period $\tau_n$ and
\begin{equation}
\label{entire-eq5}
\|u_n\|_{\alpha}+\|v_n\|_{\alpha}=\|u(\cdot, N\tau_n;0, u_n,v_n)\|_{\alpha}+\|v(\cdot, N\tau_n;0, u_n,v_n)\|_{\alpha} \leq \tilde M,
 \end{equation}
 where $N$ is such that $1\leq N\tau _n \leq 2.$

{ We claim that there is $\delta_1>0$  and $\delta_2>0$ such that
\vspace{-0.05in}\begin{equation}
\label{entire-eq6}
\|u_n\|_{\infty} \ge \delta_1 \quad \forall\,\,  n\ge 1.
\vspace{-0.05in}\end{equation}
and
\vspace{-0.05in}\begin{equation}
\label{entire-eq6bis}
\|v_n\|_{\infty} \ge \delta_2 \quad \forall\,\,  n\ge 1.
\vspace{-0.05in}\end{equation}
Since the proof of \eqref{entire-eq6} and \eqref{entire-eq6bis} are similar, we only prove \eqref{entire-eq6}.
 Suppose by contradiction that  \eqref{entire-eq6} does not hold. Then there exists $n_k $  such that $\|u_{n_k}\|_{\infty} < \frac{1}{n_k}$ for every  $k \ge 1$.  Let  $k_0$ such that $\frac{1}{n_{k}}<\delta_0$ for all $k\ge k_0.$  By Lemma \ref{persistence-lm1} and the proof of Lemma \ref{persistence-lm2}, we get that $u(\cdot,t;0,u_{n_{k}},v_{n_{k}}) \ge u(t;\inf u_{n_k})$ for all $t>0$ and $k\ge k_0,$ where $u(t;\inf u_{n_k})$ is the solution of
$$
 u_t= u\Big(a_{0,\inf}-a_{2,\sup}\bar A_2-\frac{\chi_1 k}{d_3}\epsilon_0-(a_{1,\sup}-\frac{\chi_1 k}{d_3}) u\Big)
$$
with $u(0;\inf u_{n_k})=\inf u_{n_k}$.
Let $\delta_*=\frac{ a_{0,\inf}-a_{2,\sup}\bar A_2-\frac{\chi_1 k}{d_3}\epsilon_0}{2(a_{1,\sup}-\frac{\chi_1 k}{d_3})}$ and choose $k$ large enough such that $\frac{1}{n_k}<\delta_*$.
 There is  $t_0>0$ such that $ u(t;\inf u_{n_k})>\delta_*$  for all $t\ge t_0$.  Then we have
\vspace{-0.05in} $$
 u_{n_k}(x)=u(\cdot,m \tau_{n_k};0,u_{n_k},u_{n_k})\ge u(m \tau_{n_k};\inf u_{n_k})>\delta^*
\vspace{-0.05in} $$
 for all $m\in \NN$ satisfying that $m\tau_{n_k}>t_0$. This is a contradiction. Therefore,
 \eqref{entire-eq6} holds.

 By \eqref{entire-eq5} and Arzela-Ascoli theorem, there exist $\{n_k\}$, $(u^{**},v^{**}) \in C^0(\bar{\Omega})\times C^0(\bar{\Omega})$ such that $(u_{n_k},u_{n_k})$  converges to $(u^{**},v^{**})$ in $C^0(\bar{\Omega})\times C^0(\bar{\Omega})$. By \eqref{entire-eq6} and \eqref{entire-eq6bis}, we have that
 $\|u^{**}(\cdot)\|_{\infty}\ge \delta_1 $ and $\|v^{**}(\cdot)\|_{\infty}\ge \delta_2.$}
We claim that $(u(\cdot,t;0,u^{**},v^{**}),v(\cdot,t;0,u^{**},v^{**})$, $w(\cdot,t;0,u^{**},v^{**}))$
 is a steady state solution of \eqref{u-v-w-eq00}, that is,
\begin{equation}
\label{entire-eq7}
u(\cdot,t;0,u^{**},v^{**})=u^{**}(\cdot)\quad \text{and}\quad v(\cdot,t;0,u^{**},v^{**})=v^{**}(\cdot)\quad \text{for all} \, t\ge 0.
\vspace{-0,05in}\end{equation}
In fact,
let $\epsilon>0$ be fix and let $t>0$. Note that
 \vspace{-0,05in}
 $$ [n_k  t]\tau_{n_k}=\frac{[n_kt]}{n_k}\leq t \leq \frac{[n_kt]+1}{n_k}=([n_k t]+1)\tau_{n_k}.
$$
Then, we can  choose $k$ large enough such that
\[|u(x,t;0,u^{**},v^{**})-u(x,t;0,u_{n_k},v_{n_k})|< \epsilon,\quad |u_{n_k}(x)-u^{**}(x)|< \epsilon,\quad |v_{n_k}(x)-v^{**}(x)|< \epsilon,\]
\[ |v(x,t;0,u^{**},v^{**})-v(x,t;0,u_{n_k},v_{n_k})|< \epsilon,\quad|v(x,\frac{[n_kt]}{n_k};0,u_{n_k},v_{n_k})-v(x,t;0,u_{n_k},v_{n_k})|< \epsilon,\]
\[ |u(x,\frac{[n_kt]}{n_k};0,u_{n_k},v_{n_k})-u(x,t;0,u_{n_k},v_{n_k})|< \epsilon.\]
for all $x\in\bar\Omega$.
We then have
\vspace{-0,05in}\begin{align*}
|u(x,t;0,u^{**},v^{**})-u^{**}|&\le |u(x,t;0,u^{**},v^{**})-u(x,t;0,u_{n_k},v_{n_k})| +|u_{n_k}(x)-u^{**}(x)|\\
&\quad +|u(x,t;0,u_{n_k},v_{n_k})-u(x,[n_k t]\tau_{n_k};0,u_{n_k},v_{n_k})|<3 \epsilon\quad \forall\,\, x\in\bar\Omega,
\end{align*}
and
\vspace{-0,05in}\begin{align*}
|v(x,t;0,u^{**},v^{**})-v^{**}|&\le |v(x,t;0,u^{**},v^{**})-v(x,t;0,u_{n_k},v_{n_k})| +|v_{n_k}(x)-v^{**}(x)|\\
&\quad +|v(x,t;0,u_{n_k},v_{n_k})-v(x,[n_k t]\tau_{n_k};0,u_{n_k},v_{n_k})|<3 \epsilon\quad \forall\,\, x\in\bar\Omega.
\end{align*}
Letting $\epsilon\to 0$, \eqref{entire-eq7} follows.

(iii) { Note that solutions of the following system,
  $$
  \begin{cases}
  u_t=u(a_0(t)-a_1(t)u-a_2(t)v)\cr
  v_t=v(b_0(t)-b_1(t)u-b_2(t)v)\cr
  0=ku(t)+lv(t)-\lambda w(t)
  \end{cases}
  $$
  are
  spatially homogeneous solutions $(u(t),v(t),w(t))$ of \eqref{u-v-w-eq00}.} By (H4) and Remark \ref{rk-2},
  \eqref{stability-cond-1-eq1} is satisfied.
   (iii) then
  follows { from} Lemma \ref{persistence-lm6}.

(2) It follows from the similar arguments as those in (1).
\end{proof}

\section{Extinction of one of the species}

In this section, our aim is to find conditions on the parameters  which guarantee
 the extinction of {  the species $u$}.
First we prove a lemma.

Assume (H1) or (H2). For given { $u_0,v_0\in C^+(\bar\Omega)$}, let
  $$L_1(t_0,u_0,v_0)=\limsup_{t \to \infty}(\max_{x \in \bar{\Omega}}u(x,t;t_0,u_0,v_0)),\,\,\, l_1(t_0,u_0,v_0)=\liminf_{t \to \infty}(\min_{x \in \bar{\Omega}}u(x,t;t_0,u_0,v_0)),$$
  and
  $$L_2(t_0,u_0,v_0)=\limsup_{t \to \infty}(\max_{x \in \bar{\Omega}}v(x,t;t_0,u_0,v_0)),\,\,\,
   l_2(t_0,u_0,v_0)=\liminf_{t \to \infty}(\min_{x \in \bar{\Omega}}v(x,t;t_0,u_0,v_0)).$$
  If no confusion occurs, we may write $L_i(t_0,u_0,v_0)$ and $l_i(t_0,u_0,v_0)$ as
  $L_i$ and $l_i$ ($i=1,2$) respectively.
By Theorem \ref{thm-global-000} we have
 $$0\leq l_1 \leq L_1 <\infty,\quad 0\leq l_2 \leq L_2 < \infty.$$
Furthermore, using the definition of $\limsup$ and  of $\liminf,$ and elliptic regularity,
we get that given $\epsilon >0,$ there exists $T_{\epsilon}>0$ such that
\begin{equation}
\label{eq-005}
l_1-\epsilon \leq u(x,t) \leq L_1+\epsilon , \quad l_2-\epsilon \leq v(x,t) \leq L_2+\epsilon,\quad \forall \,\,t>T_{\epsilon}.
\end{equation}

\begin{lemma}
\label{lem-extinction-01}
\begin{itemize}
\item[(1)] Assume $a_{1,\inf}>\frac{k \chi_1}{d_3}$ and $a_{2,\inf}\ge \frac{l \chi_1}{d_3}$. Then
\begin{equation}\label{extinction-eq-000}
L_1\leq \frac{\left\{ a_{0,\sup}-a_{2,\inf} l_2\right\}_{+}}{a_{1,\inf}-\frac{\chi_{1}k}{d_{3}}}.
\end{equation}

\item[(2)] Assume  $b_{2,\inf}>\frac{l \chi_2}{d_3}$. Then
\begin{equation}\label{extinction-eq-04}
L_2\leq \frac{\left\{ b_{0,\sup}-\frac{\chi_{2}l}{d_{3}}l_2+{\Big(b_{1,\inf}-k\frac{\chi_2}{d_3}\Big)_-L_1}\right\}_{+}}{b_{2,\inf}-\frac{\chi_{2}l}{d_{3}}},
\end{equation}
and
\begin{equation}\label{extinction-eq-05}
l_2\geq \frac{\left\{ b_{0,\inf}-{ \Big(\Big(b_{1,\sup}-k\frac{\chi_2}{d_3}\Big)_++k\frac{\chi_2}{d_3}\Big)L_1}-\frac{\chi_{2}l}{d_{3}}L_2\right\}_{+}}{b_{2,\sup}-\frac{\chi_{2}l}{d_{3}}}.
\end{equation}
\end{itemize}
\end{lemma}

\begin{proof}
(1) From the first equation of \eqref{u-v-w-eq00}, \eqref{eq-005},  and the fact that $a_{2,\inf}\geq \frac{\chi_1l}{d_3}$,  we have
\begin{align*}
&u_t-d_1\Delta u+\chi_1 \nabla u \cdot \nabla w\nonumber\\
&=u\left\{ a_0(t,x)-(a_1(t,x)-\frac{\chi_1 }{d_3}k)u -(a_2(t,x)-l\frac{\chi_1 }{d_3})v-\frac{\chi_1 }{d_3}\lambda w\right\}\nonumber\\
& \leq u\left\{ a_{0,\sup}-(a_{1,\inf}-\frac{\chi_1 }{d_3}k)u-a_{2,\inf}l_2+\left(a_{2,\sup}+k\frac{\chi_1}{d_3}\right)\epsilon\right\}
\end{align*}
for $t\ge T_\epsilon$,
and thus since $a_{1,\inf}>\frac{\chi_{1}k}{d_{3}},$ \eqref{extinction-eq-000} follows {from parabolic comparison principle.}

(2) From the second equation of \eqref{u-v-w-eq00} and \eqref{eq-005}, we have that
\begin{align*}
&v_t-d_2\Delta v+\chi_2 \nabla v \cdot \nabla w\nonumber\\
&=v\left\{ b_0(t,x)-(b_2(t,x)-\frac{\chi_2 }{d_3}k)v -(b_1(t,x)-k\frac{\chi_2 }{d_3})u-\frac{\chi_2 }{d_3}\lambda w\right\}\nonumber\\
& \leq v\left\{ b_{0,\sup}-(b_{2,\inf}-\frac{\chi_2 }{d_3}k)v+{ \Big(b_{1,\inf}-k\frac{\chi_2}{d_3}\Big)_-L_1}-l\frac{\chi_2}{d_3}l_2+\Big((k+l)\frac{\chi_{2}}{d_{3}}
+{\Big(b_{1,inf}-k\frac{\chi_2}{d_3}\Big)_-}\Big)\epsilon\right\}
\end{align*}
for $t\ge T_\epsilon$,
and   \eqref{extinction-eq-04} follows {from  parabolic comparison principle.}

Similarly, we have
\begin{align*}
&v_t-d_2\Delta v+\chi_2 \nabla v \cdot \nabla w\nonumber\\
&=v\left\{ b_0(t,x)-(b_2(t,x)-\frac{\chi_2 }{d_3}k)v -(b_1(t,x)-k\frac{\chi_2 }{d_3})u-\frac{\chi_2 }{d_3}\lambda w\right\}\nonumber\\
&\geq v\left\{ b_{0,\inf}-(b_{2,\sup}-\frac{\chi_2 }{d_3}k)v-{ \Big(b_{1,\sup}-k\frac{\chi_2}{d_3}\Big)_+L_1-k\frac{\chi_2}{d_3}L_1}-l\frac{\chi_2}{d_3}L_2
-\Big(l\frac{\chi_{2}}{d_{3}}+{\Big(b_{1,\sup}-k\frac{\chi_2}{d_3}\Big)_+}\Big)\epsilon\right\}
\end{align*}
for $t\ge T_\epsilon$,
and   \eqref{extinction-eq-05} thus follows from { parabolic comparison principle.}
\end{proof}

Now we prove Theorem \ref{thm-extinction}.

\begin{proof}[Proof of Theorem \ref{thm-extinction}]
We first  prove that $L_1=0$.

Suppose by contradiction that $L_1>0.$
Then by \eqref{extinction-eq-000} and \eqref{Asymp-exclusion-eq-00},  we have
\begin{equation}\label{Asymp-exclusion-eq-08}
l_2<\frac{a_{0,\sup}}{a_{2,\inf}}.
\end{equation}
By  \eqref{Asymp-exclusion-eq-03}, we have
 \begin{align*}
a_{2,\inf}\big(b_{0,\inf}(b_{2,\inf}-l\frac{\chi_2}{d_3})-b_{0,\sup}\frac{\chi_2}{d_3}l\big) &\geq a_{0,\sup}\big((b_{2,\inf}-l\frac{\chi_2}{d_3})(b_{2,\sup}-l\frac{\chi_2}{d_3})-(l\frac{\chi_2}{d_3})^2\big)\nonumber\\
&=a_{0,\sup}\big((b_{2,\inf}-l\frac{\chi_2}{d_3})b_{2,\sup}-l\frac{\chi_2}{d_3}b_{2,\inf}\big)\nonumber\\
&\geq a_{0,\sup}(b_{2,\inf}-2l\frac{\chi_2}{d_3})b_{2,\sup}.
\end{align*}
 This together with the fact that $a_{2,\inf}\big(b_{0,\inf}(b_{2,\inf}-l\frac{\chi_2}{d_3})-b_{0,\sup}\frac{\chi_2}{d_3}l\big)\leq a_{2,\inf}b_{0,\sup}(b_{2,\inf}-2l\frac{\chi_2}{d_3}), $ we get
$$
a_{2,\inf}b_{0,\sup}(b_{2,\inf}-2l\frac{\chi_2}{d_3}) \geq a_{0,\sup}(b_{2,\inf}-2l\frac{\chi_2}{d_3})b_{2,\sup},
$$
which combines with $b_{2,\inf}-2l\frac{\chi_2}{d_3}>0$ implies
$$
a_{2,\inf}b_{0,\sup} \geq a_{0,\sup}b_{2,\sup} \geq a_{0,\sup} 2l\frac{\chi_2}{d_3}.
$$
Therefore
 \begin{equation}
\label{Asymp-exclusion-eq-09}
b_{0,\sup}-\frac{\chi_{2}l}{d_{3}}l_2> b_{0,\sup}-\frac{\chi_{2}l}{d_{3}}\frac{a_{0,\sup}}{a_{2,\inf}}\geq 0.
\end{equation}

From \eqref{extinction-eq-05}, we get
\begin{equation*}
    \frac{ l\chi_2}{d_3}L_2\geq b_{0,\inf}-{ \Big(\Big(b_{1,\sup}-k\frac{\chi_2}{d_3}\Big)_++k\frac{\chi_2}{d_3}\Big)L_1}-(b_{2,\sup}-\frac{\chi_{2}}{d_{3}}l)l_2.
\end{equation*}
Thus, from  \eqref{extinction-eq-000}  and $L_1>0,$ we get
\begin{equation*}
   \frac{l\chi_2}{d_3}L_2\geq b_{0,\inf} -{ \Big(\Big(b_{1,\sup}-k\frac{\chi_2}{d_3}\Big)_++k\frac{\chi_2}{d_3}\Big)}\frac{\left\{ a_{0,\sup}-a_{2,\inf} l_2\right\}}{a_{1,\inf}-\frac{\chi_{1}k}{d_{3}}}-(b_{2,\sup}-\frac{\chi_{2}}{d_{3}}l)l_2.
\end{equation*}
Therefore
\begin{align*}
  \frac{l\chi_2}{d_3}(a_{1,\inf}-\frac{\chi_{1}k}{d_{3}})L_2
  &\geq b_{0,\inf}(a_{1,\inf}-\frac{\chi_{1}k}{d_{3}})- { \Big(\Big(b_{1,\sup}-k\frac{\chi_2}{d_3}\Big)_++k\frac{\chi_2}{d_3}\Big)}a_{0,\sup}\nonumber\\
  &-\Big((a_{1,\inf}-\frac{\chi_{1}k}{d_{3}})(b_{2,\sup}-\frac{\chi_{2}}{d_{3}}l)-{ \Big(\Big(b_{1,\sup}-k\frac{\chi_2}{d_3}\Big)_++k\frac{\chi_2}{d_3}\Big)}a_{2,\inf}\Big)l_2.
\end{align*}
It follows from the last inequality, \eqref{Asymp-exclusion-eq-09},  and  \eqref{extinction-eq-04} that
\begin{align*}
& \frac{l\chi_2}{d_3}(a_{1,\inf}-\frac{\chi_{1}k}{d_{3}})\frac{\Big\{ b_{0,\sup}-\frac{\chi_{2}l}{d_{3}}l_2+{\Big(b_{1,\inf}-k\frac{\chi_2}{d_3}\Big)_-L_1}\Big\}}{b_{2,\inf}-\frac{\chi_{2}l}{d_{3}}}\nonumber\\
&\geq b_{0,\inf}(a_{1,\inf}-\frac{\chi_{1}k}{d_{3}})- { \Big(\Big(b_{1,\sup}-k\frac{\chi_2}{d_3}\Big)_++k\frac{\chi_2}{d_3}\Big)}a_{0,\sup}\nonumber\\
  &-\Big((a_{1,\inf}-\frac{\chi_{1}k}{d_{3}})(b_{2,\sup}-\frac{\chi_{2}}{d_{3}}l)-{ \Big(\Big(b_{1,\sup}-k\frac{\chi_2}{d_3}\Big)_++k\frac{\chi_2}{d_3}\Big)}a_{2,\inf}\Big)l_2.
\end{align*}
Therefore from \eqref{extinction-eq-000}, we get
\begin{align*}
& \frac{l\chi_2}{d_3}(a_{1,\inf}-\frac{\chi_{1}k}{d_{3}})\frac{\Big\{ b_{0,\sup}-\frac{\chi_{2}l}{d_{3}}l_2+{\Big(b_{1,\inf}-k\frac{\chi_2}{d_3}\Big)_-\frac{\Big\{ a_{0,\sup}-a_{2,\inf} l_2\Big\}}{a_{1,\inf}-\frac{\chi_{1}k}{d_{3}}}}\Big\}}{b_{2,\inf}-\frac{\chi_{2}l}{d_{3}}}\nonumber\\
&\geq b_{0,\inf}(a_{1,\inf}-\frac{\chi_{1}k}{d_{3}})- { \Big(\Big(b_{1,\sup}-k\frac{\chi_2}{d_3}\Big)_++k\frac{\chi_2}{d_3}\Big)}a_{0,\sup}\nonumber\\
  &-\Big((a_{1,\inf}-\frac{\chi_{1}k}{d_{3}})(b_{2,\sup}-\frac{\chi_{2}}{d_{3}}l)-{ \Big(\Big(b_{1,\sup}-k\frac{\chi_2}{d_3}\Big)_++k\frac{\chi_2}{d_3}\Big)}a_{2,\inf}\Big)l_2.
\end{align*}
Thus
\begin{align}\label{ra-00001}
	& \Big\{\underbrace{(a_{1,\inf}-\frac{\chi_{1}k}{d_{3}})\Big[(b_{2,\inf}-\frac{\chi_{2}l}{d_{3}})(b_{2,\sup}-\frac{\chi_{2}l}{d_{3}})-(l\frac{\chi_2}{d_3})^2\Big]}_{B_1}\Big\}l_2\nonumber\\
&-\Big\{\underbrace{\Big[{ \Big(\Big(b_{1,\sup}-k\frac{\chi_2}{d_3}\Big)_++k\frac{\chi_2}{d_3}\Big)}(b_{2,\inf}-\frac{\chi_{2}l}{d_{3}})+\frac{l\chi_2}{d_3}{\Big(b_{1,\inf}-k\frac{\chi_2}{d_3}\Big)_-}\Big]a_{2,\inf}}_{B_2}\Big\}l_2\nonumber\\
&\geq \underbrace{\left(b_{0,\inf}(b_{2,\inf}-\frac{\chi_{2}l}{d_{3}})-l\frac{\chi_2}{d_3}b_{0,sup}\right)\left(a_{1,\inf}-\frac{\chi_{1}k}{d_{3}}\right)}_{A_1}\nonumber\\
&-\underbrace{\Big[{ \Big(\Big(b_{1,\sup}-k\frac{\chi_2}{d_3}\Big)_++k\frac{\chi_2}{d_3}\Big)}(b_{2,\inf}-\frac{\chi_{2}l}{d_{3}})+\frac{l\chi_2}{d_3}{ \Big(b_{1,\inf}-k\frac{\chi_2}{d_3}\Big)_-}\Big]a_{0,\sup}}_{A_2}
\end{align}
Then, inequality \eqref{ra-00001} is equivalent to
\begin{equation}\label{ra-aa1}
B l_2\geq A,
\end{equation}
with $B=B_1-B_2$ and $A=A_1-A_2$. Note that  \eqref{Asymp-exclusion-eq-04} yields that   $A>0.$ This combined with \eqref{ra-aa1} implies  that   $B>0$.  Therefore, inequality \eqref{ra-aa1} becomes
$$
l_2\geq \frac{A}{B}.
$$
Then thanks  to equation \eqref{Asymp-exclusion-eq-08}, we get
$$
B>\frac{a_{2,\inf}}{a_{0,\sup}}A.
$$
That means
\begin{align*}
 &a_{0,\sup}(a_{1,\inf}-\frac{\chi_{1}k}{d_{3}})\Big[(b_{2,\inf}-\frac{\chi_{2}l}{d_{3}})(b_{2,\sup}-\frac{\chi_{2}l}{d_{3}})-(l\frac{\chi_2}{d_3})^2\Big]\nonumber\\
 &-a_{0,\sup}\Big[{ \Big(\Big(b_{1,\sup}-k\frac{\chi_2}{d_3}\Big)_++k\frac{\chi_2}{d_3}\Big)}(b_{2,\inf}-\frac{\chi_{2}l}{d_{3}})+\frac{l\chi_2}{d_3}{\Big(b_{1,\inf}-k\frac{\chi_2}{d_3}\Big)_-}\Big]a_{2,\inf}\nonumber\\
&>a_{2,\inf}\left(b_{0,\inf}(b_{2,\inf}-\frac{\chi_{2}l}{d_{3}})-l\frac{\chi_2}{d_3}b_{0,sup}\right)\left(a_{1,\inf}-\frac{\chi_{1}k}{d_{3}}\right)\nonumber\\
&-\Big[{ \Big(\Big(b_{1,\sup}-k\frac{\chi_2}{d_3}\Big)_++k\frac{\chi_2}{d_3}\Big)}(b_{2,\inf}-\frac{\chi_{2}l}{d_{3}})+\frac{l\chi_2}{d_3}{ \Big(b_{1,\inf}-k\frac{\chi_2}{d_3}\Big)_-}\Big]a_{0,\sup}a_{2,\inf}.
\end{align*}
Thus
$$
a_{0,\sup}\Big[(b_{2,\inf}-\frac{\chi_{2}l}{d_{3}})(b_{2,\sup}-\frac{\chi_{2}l}{d_{3}})-(l\frac{\chi_2}{d_3})^2\Big]>a_{2,\inf}\left(b_{0,\inf}
(b_{2,\inf}-\frac{\chi_{2}l}{d_{3}})-l\frac{\chi_2}{d_3}b_{0,\sup}\right),
$$
which contradicts to  \eqref{Asymp-exclusion-eq-03}. Hence $L_1=0$.

Next, we prove \eqref{MainAsym-eq-002} and \eqref{MainAsym-eq-003}.  Since  $L_1=0,$ we get from  \eqref{extinction-eq-04} and \eqref{extinction-eq-05} respectively  that
\begin{equation}\label{Asymp-exclusion-eq-010}
L_2\leq \frac{ b_{0,\sup}-\frac{\chi_{2}l}{d_{3}}l_2}{b_{2,\inf}-\frac{\chi_{2}l}{d_{3}}}.
\end{equation}
 and
\begin{equation}\label{Asymp-exclusion-eq-011}
l_2\geq \frac{ b_{0,\inf}-\frac{\chi_{2}l}{d_{3}}L_2}{b_{2,\sup}-\frac{\chi_{2}l}{d_{3}}}.
\end{equation}
Then \eqref{MainAsym-eq-002} follows from \eqref{Asymp-exclusion-eq-010} and \eqref{Asymp-exclusion-eq-011}. Furthermore \eqref{MainAsym-eq-003} follows from \eqref{MainAsym-eq-001}, \eqref{MainAsym-eq-002} and elliptic comparison principle.

Finally, assume that \eqref{v-w-eq00} has a unique positive entire solution $(v^*(x,t;\tilde b_0,\tilde b_2),w^*(x,t;\tilde b_0,\tilde b_2))$ for any
$(\tilde b_0,\tilde b_2)\in H(b_0,b_2)$.  We claim that \eqref{MainAsym-eq-004} holds.
Indeed, if  \eqref{MainAsym-eq-004} does not hold. Then there are $\tilde \epsilon_0>0$ and $t_n\to\infty$ such that
$$
\|v(\cdot,t_n+t_0;t_0,u_0,v_0)-v^*(\cdot,t_n+t_0;b_0,b_2)\|_\infty\ge \tilde \epsilon_0\quad \forall\,\, n=1,2,\cdots.
$$
Without loss of generality, we may assume that
$$
\lim_{n\to\infty} (b_0(t+t_n+t_0,x),b_2(t+t_n+t_0,x))=(\tilde b_0(t,x),\tilde b_2(t,x))
$$
and
$$
\lim_{n\to\infty} (u(x,t+t_n+t_0;t_0,u_0,v_0),v(x,t+t_n+t_0;t_0,u_0,v_0),w(x,t+t_n+t_0;t_0,u_0,v_0))=(0,\tilde v(x,t),w(x,t))
$$
locally uniformly in $(t,x)\in\RR\times \bar{\Omega}$. Then $(\tilde v(x,t),\tilde w(x,t))$ is a positive entire solution of \eqref{v-w-eq00}
and
$$
\|\tilde v(\cdot;0)-v^*(\cdot,0;\tilde b_0,\tilde b_2)\|_\infty\ge \tilde \epsilon_0,
$$
which is a contradiction. Hence \eqref{MainAsym-eq-004} holds.
\end{proof}

\end{document}